\documentclass[reqno,12pt]{amsart}
\usepackage{amsmath, amsthm, amssymb}
\usepackage{mathrsfs}

\advance \topmargin by -\headheight
\advance \topmargin by -\headsep
     
\evensidemargin \oddsidemargin
\newtheorem{theorem}{Theorem}[section]
\newtheorem{lemma}[theorem]{Lemma}

\theoremstyle{definition}

\theoremstyle{remark}

\numberwithin{equation}{section}

\def\bfx{{\mathbf x}}

\def\dbZ{{\mathbb Z}}

\def\grm{{\mathfrak m}}

\def\grs{{\mathfrak s}}

\def\alp{{\alpha}}

\def\Ups{{\Upsilon}}

\def\d{{\partial}}
\def\eps{\varepsilon}

\def\le{\leqslant} \def\ge{\geqslant}

\def\d{{\,{\rm d}}}

\begin{document}
\title[Diagonal quartic forms]{Pairs of diagonal quartic forms:\\ the asymptotic formulae}
\author[J\"org Br\"udern]{J\"org Br\"udern}
\address{Mathematisches Institut, Bunsenstrasse 3--5, D-37073 G\"ottingen, Germany}
\email{jbruede@gwdg.de}
\author[Trevor D. Wooley]{Trevor D. Wooley}
\address{Department of Mathematics, Purdue University, 150 N. University Street, West 
Lafayette, IN 47907-2067, USA}
\email{twooley@purdue.edu}
\thanks{First Author supported by Deutsche Forschungsgemeinschaft Project Number 
255083470. Second author supported by NSF grants DMS-1854398 and DMS-2001549.}
\date{}

\begin{abstract} We establish an asymptotic formula for the number of integral solutions of 
bounded height for pairs of diagonal quartic equations in $26$ or more variables. In certain 
cases, pairs in $25$ variables can be handled. 
\end{abstract}
\maketitle

\section{Introduction} 
Once again we are concerned with the pair of Diophantine equations
\begin{equation}\label{1.1}
a_1x_1^4+a_2x_2^4+\ldots +a_sx_s^4=b_1x_1^4+b_2x_2^4+\ldots +b_sx_s^4=0,
\end{equation}
wherein the given coefficients $a_j,b_j$ satisfy $(a_j,b_j)\in \dbZ^2\setminus \{(0,0)\}$ 
$(1\le j\le s)$. While our focus was on the validity of the Hasse principle for such pairs in 
two precursors of this article \cite{Jems1,BW21}, we now investigate the asymptotic 
density of integral solutions. Denote by $\mathscr N(P)$ the number of solutions in integers 
$x_j$ with $|x_j|\le P$ $(1\le j\le s)$ to this system. Then, subject to a natural rank 
condition on the coefficient matrix, one expects an asymptotic formula for $\mathscr N(P)$ 
to hold provided that $s$ is not too small. Indeed, following Hardy and Littlewood \cite{PN3} 
in spirit, the quantity $P^{8-s}\mathscr N(P)$ should tend to a limit that is itself a product 
of local densities. On a formal level, the densities are readily described. The real density, 
also known as the singular integral, is defined by
\begin{equation}
\label{1.2} \mathfrak I = \lim_{T\to \infty} \int_{-T}^T\int_{-T}^T\prod_{j=1}^s 
\int_{-1}^1 e\big((a_j\alpha+b_j\beta)t_j^4\big)\,\mathrm d t_j\,\mathrm d\alpha\,
\mathrm d\beta
\end{equation}
whenever the limit exists. Let $M(q)$ denote the number of solutions $\bfx$ in 
$(\dbZ/q\dbZ)^s$ satisfying (\ref{1.1}). Then for primes $p$, the $p$-adic density is 
defined by
\begin{equation}\label{1.3}
\mathfrak s_p = \lim_{h\to\infty} p^{(2-s)h} M(p^h),
\end{equation}
assuming again that this limit exists. In case of convergence, the product 
$\mathfrak S = \prod_p \mathfrak s_p$ is referred to as the singular series, and the 
desired asymptotic relation can be presented as the limit formula
\begin{equation}\label{1.4}
\lim_{P\to\infty} P^{8-s} \mathscr N (P) = \mathfrak I \mathfrak S.
\end{equation} 

Note that \eqref{1.4} can hold only when in each of the two equations comprising 
\eqref{1.1} there are sufficiently many non-zero coefficients. Of course one may pass from 
\eqref{1.1} to an equivalent system obtained by taking linear combinations of the two 
constituent equations. Thus, the invariant $q_0=q_0(\mathbf a,\mathbf b)$, defined by
$$q_0(\mathbf a,\mathbf b)=\min_{(c,d)\in\mathbb Z^2\setminus\{(0,0)\}}\text{card}
\{1\le j \le s: ca_j+db_j\neq 0\},$$
must be reasonably large. Indeed, it follows from Lemmata 3.1, 3.2 and 3.3 in our 
companion paper \cite{BW21} that the conditions $s\ge 16$ and $q_0\ge 12$ ensure that 
the limits \eqref{1.2} and \eqref{1.3} all exist, that the product $\mathfrak S$ is absolutely 
convergent, and that the existence of non-singular solutions to the system \eqref{1.1} in 
each completion of the rationals implies that $\mathfrak I \mathfrak S>0$. A first result 
concerning the limit \eqref{1.4} is then obtained by introducing the moment estimate
\begin{equation}\label{1.5}
\int_0^1\bigg| \sum_{x\le P}e(\alpha x^4)\bigg|^{14}\,\mathrm d\alpha \ll 
P^{10+\varepsilon},
\end{equation}
derived as the special case $u=14$ of Lemma \ref{lemma5.3} below, to a familiar method 
of Cook \cite{Coo1972} (see also \cite{BrC}). Here we point out that the estimate 
(\ref{1.5}) first occurs implicitly in the proof of \cite[Theorem 4.1]{Woo2012}, conditional 
on the validity of the (now proven) main conjecture in Vinogradov's mean value theorem 
(for which see \cite{BDG2016} and \cite[Corollary 1.3]{Woo2019}). In this way, one 
routinely confirms \eqref{1.4} when $s\ge 29$ and $q_0\ge 15$. This result, although not 
explicitly mentioned in the literature, is certainly familiar to experts in the area, and has to 
be considered as the state of the art today. It seems worth remarking in this context that, 
at a time when the estimate \eqref{1.5} was not yet available, the authors 
\cite{BWBull,Camb} handled the case $s\ge 29$ with more restrictive rank conditions. The 
main purpose of this memoir is to make three variables redundant. 

\begin{theorem}\label{theorem1.1} For pairs of equations \eqref{1.1} with $s\ge 26$ and 
$q_0\ge 15$, one has $\mathscr N(P)\sim \mathfrak I\mathfrak SP^{s-8}$.
\end{theorem}

Relaxing the rank condition $q_0\ge 15$ appears to be a difficult enterprise, as we now 
explain. Consider a pair of equations \eqref{1.1} with $s\ge 29$, and suppose that 
$b_i=a_j=0$ for $1\le i\le 14<j\le s$. These two equations are independent and thus 
$\mathscr N(P)$ factorises as $\mathscr N(P)=N_1(P)N_2(P)$, where $N_1(P)$ and 
$N_2(P)$ denote the number of integral solutions of the respective single equations
\begin{equation}
\label{1.6}
a_1x_1^4+a_2x_2^4+\ldots +a_{14}x_{14}^4=0,
\end{equation}
with $|x_j|\le P$ $(1\le j\le 14)$, and
\begin{equation}\label{1.7}
b_{15}y_1^4+b_{16}y_2^4+\ldots +b_sy_{s-14}^4=0,
\end{equation}
with $|y_j|\le P$ $(1\le j\le s-14)$. The equation (\ref{1.7}) has at least $15$ non-zero 
coefficients, and so a straightforward application of the Hardy-Littlewood method using the 
mean value \eqref{1.5} shows that $P^{18-s}N_2(P)$ tends to a limit as $P\to \infty$, with 
this limit equal to a product of local densities analogous to $\mathfrak I$ and 
$\mathfrak s_p$. By choosing $b_j=(-1)^j$ for $15\le j\le s$, we ensure that this limit is 
positive, and thus $P^{8-s}\mathscr N(P)$ tends to a limit as $P\to\infty$ if and only if 
$P^{-10}N_1(P)$ likewise tends to a limit. From the definitions \eqref{1.2} and 
\eqref{1.3}, it is apparent that the local densities $\mathfrak I$ and $\grs_p$ factorise 
into components stemming from the equations underlying $N_1$ and $N_2$. The relation 
\eqref{1.4} therefore holds for this particular pair of equations if and only if 
$P^{-10}N_1(P)$ tends to the product of local densities associated with the equation 
\eqref{1.6}. In particular, were \eqref{1.4} known to hold in any case where $q_0=14$ and 
$s$ is large, then it would follow that $P^{-10}N_1(P)$ tends to the limit 
suggested by a formal application of the circle method, a result that is not yet known. This 
shows that relaxing the condition on $q_0$ would imply progress with single diagonal 
quartic equations.\par

The invariant $q_0$ is a very rough measure for the entanglement of the two equations 
present in \eqref{1.1}. This can be refined considerably. The pairs $(a_j,b_j)$ are all 
non-zero in $\mathbb Z^2$, so they define a point $(a_j:b_j)\in\mathbb P(\mathbb Q)$. We 
refer to indices $i,j\in\{1,2,\ldots,s\}$ as {\em equivalent} if $(a_i:b_i)=(a_j:b_j)$. This 
defines an equivalence relation on $\{1,2,\ldots,s\}$. Suppose that there are $\nu$ 
equivalence classes with $r_1,\ldots ,r_\nu$ elements, respectively, where 
$r_1\ge r_2\ge \ldots\ge r_\nu$. On an earlier occasion \cite{Camb} we named the tuple 
$(r_1,\ldots,r_\nu)$ the {\em profile} of the equations \eqref{1.1}. Note that $q_0=s-r_1$, 
whence our assumed lower bound $q_0\ge 15$ implies that $r_1\le s-15$ and $\nu\ge 2$. 
If more is known about the profile, then we can save yet another variable.
 
\begin{theorem}\label{theorem1.2} Suppose that $s= 25$ and that $(r_1,\ldots,r_\nu)$ is 
the profile of the pair of equations \eqref{1.1}. If $q_0\ge 16$ and $\nu\ge 5$, then 
$\mathscr N(P)\sim \mathfrak I\mathfrak SP^{s-8}$.
\end{theorem} 

For a pair \eqref{1.1} in ``general position'' one has $\nu=s$ and $r_1=1$, and in a 
quantitative sense easily made precise, such pairs constitute almost all such Diophantine 
systems. Hence, the conclusion of Theorem \ref{theorem1.2} applies to almost all pairs 
of equations of the shape \eqref{1.1}.\par

We pointed out long ago \cite{Camb} that a diffuse profile can be advantageous. However, 
even with the estimate \eqref{1.5} in hand, the method of \cite{Camb} only handles cases 
where $s\ge 27$ and $r_1$ and $r_2$ are not too large. Thus our results improve on all 
previous work on the subject even if the input to the published versions is enhanced by the 
newer mean value bound \eqref{1.5}.\par

It is time to describe the methods, and in particular the new ideas involved in the proofs. 
Our more recent results specific to systems of diagonal quartic forms 
\cite{Jems1,Jems2,BW21} all depend on large values estimates for Fourier coefficients of 
powers of Weyl sums, and the current communication is no exception. The large values 
estimates provide upper bounds for higher moments of these Fourier coefficients, and these 
in turn yield mean value bounds for correlations of Weyl sums. We describe this link here in 
a setting appropriate for application to pairs of equations. Consider a $1$-periodic twice 
differentiable function $h:\mathbb R \to\mathbb R$. Its Fourier expansion
\begin{equation}\label{Fou}
h(\alpha)=\sum_{n\in \dbZ}\hat h(n)e(\alpha n)
\end{equation}  
converges uniformly and absolutely. Hence, by orthogonality, one has
\begin{equation}\label{3rd}
\int_0^1\!\!\int_0^1 h(\alpha)h(\beta)h(-\alpha-\beta)\,\mathrm d\alpha\,\mathrm d\beta = 
\sum_{n\in \dbZ}\hat h(n)^3.
\end{equation}
The methods of \cite{Jems1,Jems2,BW21} rest on this and closely related identities, 
choosing $h(\alpha)=|g(\alpha)|^u$ with suitable quartic Weyl sums $g$ and a positive real 
number $u$. As a service to future scholars, we analyse in some detail the differentiability 
properties of functions like $|g(\alpha)|^u$ in \S3. It transpires that when $u\ge 2$ then 
the relation \eqref{3rd} holds. We use \eqref{3rd} with $h(\alpha)=|f(\alpha)|^u$, where 
now 
\begin{equation}\label{ff}
f(\alpha)= \sum_{x\le P} e(\alpha x^4)
\end{equation}    
is the ordinary Weyl sum. We then obtain new entangled mean value estimates for smaller 
values of $u$. This alone is not of strength sufficient to reach the conclusions of Theorem 
\ref{theorem1.1}.\par

As experts in the field will readily recognise, for larger values of $u$ the quality of the 
aforementioned mean value estimates is diluted by major arc contributions, and one would 
therefore like to achieve their removal. Thus, if $\mathfrak n$ is a $1$-periodic set of real 
numbers with $\mathfrak n\cap [0,1)$ a classical choice of minor arcs and 
${\bf 1}_{\mathfrak n}$ is the indicator function of $\mathfrak n$, then one is tempted to 
apply the function $h(\alpha)={\bf 1}_{\mathfrak n}(\alpha)|f(\alpha)|^u$ in place of 
$|f(\alpha)|^u$ within \eqref{3rd}. However, this function is no longer continuous. We 
bypass this difficulty by introducing a smoothed Farey dissection in \S4. This is achieved by a 
simple and very familiar convolution technique that should be useful in other contexts, too. 
In this way, in \S5 we obtain a minor arc variant of the cubic moment method developed in 
our earlier work \cite{Jems1}. Equipped with this and the mean value bounds that follow 
from it, one reaches the conclusions of Theorem \ref{theorem1.1} in the majority of cases 
under consideration. Unfortunately, some cases with exceptionally large values of $r_j$ 
stubbornly deny treatment. To cope with these remaining cases, we develop a mixed 
moment method in \S6.\par

The point of departure is a generalisation of \eqref{3rd}. If $h_1,h_2,h_3$ are functions 
that qualify for the discussion surrounding \eqref{Fou} and \eqref{3rd}, then by invoking 
orthogonality once again, we see that
\begin{equation}\label{3rdv}
\int_0^1\!\!\int_0^1 h_1(\alpha)h_2(\beta)h_3(-\alpha-\beta)\,\mathrm d\alpha\,\mathrm 
d\beta = \sum_{n\in \mathbb Z}\hat h_1(n)\hat h_2(n) \hat h_3(n).
\end{equation}
By H\"older's inequality, the right hand side here is bounded in terms of the three moments
\begin{equation}\label{3rdv2}
\sum_{n\in \mathbb Z}|\hat h_j(n)|^3.
\end{equation}
In all cases where $h_j(\alpha)= |f(\alpha)|^{u_j}$ for some even positive integral 
exponent $u_j$ one has $\hat h_j(n)\ge 0$, so \eqref{3rd} can be used in reverse to 
interpret \eqref{3rdv2} in terms of the number of solutions of a pair of Diophantine 
equations. The purely analytic description of the method has several advantages. First and 
foremost, one can break away from even numbers $u_j$, and still estimate all three cubic 
moments \eqref{3rdv2}. This paves the way to a complete treatment of pairs of equations 
\eqref{1.1} with $s\ge 26$ and $q_0\ge 15$. Beyond this, the identity \eqref{3rdv} offers 
extra flexibility for the arithmetic harmonic analysis. Instead of the homogeneous passage 
from \eqref{3rdv} to \eqref{3rdv2} one could apply H\"older's inequality with differing 
weights. As an example of stunning simplicity, we note that the expression in \eqref{3rdv} is 
bounded above by
$$  \biggl(\sum_{n\in \mathbb Z}|\hat h_1(n)|^2\biggr)^{1/2}
\biggl(\sum_{n\in \mathbb Z}|\hat h_2(n)|^4\biggr)^{1/4}
\biggl(\sum_{n\in \mathbb Z}|\hat h_3(n)|^4\biggr)^{1/4}. $$
If we apply this idea with $h_j(\alpha)=|f(\alpha)|^{u_j}$ and $u_j$ a positive even integer, 
then the first factor relates to a single diagonal Diophantine equation while the other two 
factors concern systems consisting of three diagonal Diophantine equations. This argument is 
dual (in the sense that we work with Fourier coefficients) to a method that we described as 
{\em complification} in our work on systems of cubic forms \cite{MA}. There is, of course, 
an obvious generalisation of \eqref{3rd} to higher dimensional integrals that has been used 
here. This points to a complex interplay between systems of diagonal equations in which the 
size parameters (number of variables and number of equations) vary, and need not be 
restricted to natural numbers. We have yet to explore the full potential of this observation.

\par We briefly comment on the role of the Hausdorff-Young inequality 
\cite[Chapter XII, Theorem 2.3]{Z} within this circle of ideas. In the notation of 
\eqref{3rdv} this asserts that
$$ \sum_{n\in \mathbb Z}|\hat h_j(n)|^3 \le \biggl(\int_0^1 |h_j(\alpha)|^{3/2}
\,\mathrm d\alpha \biggr)^2. $$
Passing through \eqref{3rdv} and \eqref{3rdv2}, one then arrives at the estimate
\begin{equation}\label{HY2} \biggl|\int_0^1\!\!\int_0^1 
h_1(\alpha)h_2(\beta)h_3(-\alpha-\beta)\,\mathrm d\alpha\,\mathrm d\beta\biggr| 
\le \prod_{j=1}^3 \biggl(\int_0^1 |h_j(\alpha)|^{3/2}\,\mathrm d\alpha \biggr)^{2/3}. 
\end{equation}
However, by H\"older's inequality, one finds 
$$ \biggl|\int_0^1\!\!\int_0^1 h_1(\alpha)h_2(\beta)h_3(-\alpha-\beta)\,\mathrm d\alpha\,
\mathrm d\beta\biggr| \le \prod_{1\le i<j\le 3} \biggl(\int_0^1 |h_i h_j|^{3/2}\,\mathrm 
d\alpha  \, \mathrm d\beta \biggr)^{1/3}, $$
where, on the right hand side, one should read $h_1=h_1(\alpha)$, $h_2=h_2(\beta)$ and 
$h_3=h_3(-\alpha-\beta)$. By means of obvious linear substitutions, this also delivers the 
bound \eqref{HY2}. This last method is essentially that of Cook \cite{Coo1972}. Our 
approach is superior because the methods are designed to remember the arithmetic source 
of the Weyl sums when estimating moments of Fourier coefficients.\par

The proof of Theorem \ref{theorem1.2} requires yet another tool that is a development of 
our multidimensional version of Hua's lemma \cite{BWBull}. This somewhat outdated work is 
based on Weyl differencing. An analysis of the method shows that whenever a new block of 
differenced Weyl sums enters the recursive process, a new entry $r_j$ to the profile of the 
underlying Diophantine system is needed. It is here where one imports undesired 
constraints on the profile, as in Theorem \ref{theorem1.2}. However, powered with the new 
upper bound \eqref{1.5}, the method just described yields a bound for a two-dimensional 
entangled mean value over eighteen Weyl sums that outperforms the cubic moments 
technique by a factor $P^{1/6}$ (compare Theorem \ref{thm5.1} with Theorem 
\ref{thm6.2}). Within a circle method approach, this mean value is introduced via H\"older's 
inequality. In the complementary factor, we have available an abundance of Weyl sums. 
Fortunately the cubic moments technique restricted to minor arcs presses the method home. 
We point out that our proof of Theorem \ref{theorem1.2} constitutes the first instance in 
which the cubic moments technique is successfully coupled with the differencing techniques 
derived from \cite{BWBull}.\par

One might ask whether more restrictive conditions on the profile allow one to reduce the 
number of variables even further. As we demonstrate at the very end of this memoir it is 
indeed possible to accelerate the convergence in \eqref{1.4}, but even the extreme 
condition $r_1=1$ seems insufficient to save a variable without another new idea.\par

Once the new moment estimates are established, our proofs of Theorems \ref{theorem1.1} 
and \ref{theorem1.2} are fairly concise. There are two reasons. First, we may import the 
major arc work, to a large extent, from \cite{BW21}. Second, more importantly, our minor 
arc treatment rests on a new inequality (Lemma \ref{lemA3} below) that entirely avoids 
combinatorial difficulties associated with exceptional profiles. This allows us to reduce the 
minor arc work to a single profile with a certain maximality property. We expect this 
argument to become a standard preparation step in related work, and have therefore 
presented this material in broad generality. We refer to \S2 where the reader will also find 
comment on previous attempts in this direction.\par
 
{\em Notation}. Our basic parameter is $P$, a sufficiently large real number. Implicit 
constants in Vinogradov's familiar symbols $\ll$ and $\gg$ may depend on $s$ and 
$\varepsilon$ as well as ambient coefficients such as those in the system \eqref{1.1}. 
Whenever $\varepsilon$ appears in a statement we assert that the statement holds for 
each positive real value assigned to $\varepsilon$. As usual, we write $e(z)$ for 
$e^{2\pi iz}$.

\section{Some inequalities} This section belongs to real analysis. We discuss a number of 
inequalities for products. As is familiar for decades, in an attempt to prove results of the 
type described in Theorems \ref{theorem1.1} and \ref{theorem1.2} via harmonic analysis, 
it is desirable to simplify to a situation where the profile is extremal relative to the conditions 
in hand, that is, the multiplicities $r_1, r_2, \ldots$ are as large as possible, and 
consequently $\nu$ is as small as is possible. In the past, most scholars have applied 
H\"older's inequality to achieve this objective, often by an {\em ad hoc} argument that led 
to the consideration of several cases separately. The purpose of this section is to make 
available general inequalities that encapsulate the reduction step in a single lemma of 
generality sufficient to include all situations that one encounters in practice.\par

The germ of our method is a classical estimate, sometimes referred to as Young's 
inequality: if $p$ and $q$ are real numbers with $p>1$ and
$$\frac{1}{p}+\frac{1}{q}=1,$$
then for all non-negative real numbers $u$ and $v$ one has
\begin{equation}
\label{HY} uv \le \frac{u^p}{p}+\frac{v^q}{q}.
\end{equation}
This includes the case $r=2$ of the bound
\begin{equation}
\label{T} |z_1z_2\cdots z_r| \le \frac1r \big(|z_1|^r+\ldots+|z_r|^r\big)
\end{equation}
which holds for all $r\in\mathbb N$ and all $z_j\in\mathbb C$ $(1\le j\le r)$. Indeed, the 
general case of \eqref{T} follows from \eqref{HY} by an easy induction on $r$.\par

In the following chain of lemmata we are given a number $\nu\in\mathbb N$ and integral 
exponents $m_j$, $M_j$ $(1\le j\le \nu)$ with 
\begin{equation}\label{Hyp1}
m_1\ge m_2\ge \ldots \ge m_\nu\ge 0, \qquad M_1\ge M_2\ge\ldots\ge M_\nu\ge 0
\end{equation}
and
\begin{equation}\label{Hyp2}
\sum_{l=1}^L m_l \le \sum_{l=1}^L M_l \quad (1\le L<\nu), \qquad 
\sum_{l=1}^\nu m_l = \sum_{l=1}^\nu M_l.
\end{equation}
We write $S_\nu$ for the group of permutations on $\nu$ elements. We refer to a function 
$w:S_\nu\to [0, 1]$ with
$$ \sum_{\sigma\in S_\nu} w(\sigma) = 1 $$
as a {\em weight on} $S_\nu$.

\begin{lemma}\label{lemA1}
Suppose that the exponents $m_j$, $M_j$ $(1\le j\le \nu)$ satisfy \eqref{Hyp1} and 
\eqref{Hyp2}. Then there is a weight $w$ on $S_\nu$ with the property that for all 
non-negative real numbers $u_1,u_2,\ldots,u_\nu$ one has
\begin{equation}\label{A1}
u_1^{m_1}u_2^{m_2}\cdots u_\nu^{m_\nu} \le \sum_{\sigma\in S_\nu} w(\sigma)
u_{\sigma(1)}^{M_1}u_{\sigma(2)}^{M_2}\cdots u_{\sigma(\nu)}^{M_\nu}.
\end{equation}
\end{lemma}

\begin{proof} We define
$$ D= \sum_{l=1}^\nu |M_l-m_l| $$
and proceed by induction on $\nu+D$. In the base case of the induction one has 
$\nu+D=1$. In this situation $\nu=1$ and $D=0$, and the claim of the lemma is trivially 
true with $\sigma=\text{id}$ and $w(\sigma)=1$.\par

Now suppose that $\nu+D>1$. We consider two cases. First we suppose that there is a 
number $\nu_1$ with $1\le \nu_1<\nu$ and
$$  \sum_{l=1}^{\nu_1} m_l = \sum_{l=1}^{\nu_1} M_l.$$
We put
$$ D_1= \sum_{l=1}^{\nu_1} |M_l-m_l|, \quad D_2 = \sum_{l=\nu_1+1}^\nu |M_l-m_l|, 
\quad \nu_2=\nu-\nu_1. $$
Then \eqref{Hyp1} and \eqref{Hyp2} are valid with $\nu_1$ in place of $\nu$, and one has 
$D_1\le D$. Hence $\nu_1+D_1<\nu +D$ so that we may invoke the inductive hypothesis 
to find a weight $w_1$ on $ S_{\nu_1}$ with
\begin{equation}\label{A3}
 u_1^{m_1}u_2^{m_2}\cdots u_{\nu_1}^{m_{\nu_1}} \le \sum_{\sigma\in S_{\nu_1}} 
w_1(\sigma)u_{\sigma(1)}^{M_1}u_{\sigma(2)}^{M_2}\cdots 
u_{\sigma(\nu_1)}^{M_{\nu_1}}.
\end{equation}
Similarly, in the current situation, the numbers $m_{\nu_1+j}$,  $M_{\nu_1+j}$ 
$(1\le j \le \nu_2)$ may take the roles of $m_j$, $M_j$ in \eqref{Hyp1} and \eqref{Hyp2} 
with $\nu_2$ in place of $\nu$. Again, we have $\nu_2+D_2<\nu+D$. Now writing $\tau$ 
for a permutation in $S_{\nu_2}$ acting on the set $\{\nu_1+1,\nu_1+2,\ldots, \nu\}$, we 
may invoke the inductive hypothesis again to find a weight $w_2$ on $S_{\nu_2}$ with
\begin{equation}\label{A4}
 u_{\nu_1+1}^{m_{\nu_1+1}}u_{\nu_1+2}^{m_{\nu_1+2}}\cdots u_{\nu}^{m_{\nu}} 
\le \sum_{\tau\in S_{\nu_2}}w_2(\tau)u_{\tau(\nu_1+1)}^{M_{\nu_1+1}}
u_{\tau(\nu_1+2)}^{M_{\nu_1+2}}\cdots u_{\tau(\nu)}^{M_{\nu}}.
\end{equation}
We multiply the inequalities \eqref{A3} and \eqref{A4}. It is then convenient to read 
permutations $\sigma$ on $1,2,\ldots,\nu_1$ and $\tau$ on $\nu_1+1,\nu_1+2,\ldots,\nu$ 
as permutations on $1,2,\ldots, \nu$ with $\sigma(j)=j$ for $j>\nu_1$ and $\tau(j)=j$ for 
$j\le \nu_1$. Then, for permutations of the type $\sigma\tau$ in $S_\nu$ we put 
$w(\sigma\tau) = w_1(\sigma)w_2(\tau)$, and we put $w(\phi)=0$ for the remaining 
permutations $\phi\in S_\nu$. With this function $w$ the product of \eqref{A3} and 
\eqref{A4} becomes \eqref{A1}, completing the induction in the case under consideration.

\par In the complementary case we have
\begin{equation}\label{A5} \sum_{l=1}^L m_l < \sum_{l=1}^L M_l \qquad (1\le L<\nu).
\end{equation}
In particular, this shows that $m_1<M_1$. Also, by comparing the case $L=\nu-1$ of 
\eqref{A5} with the equation corresponding to the case $L=\nu$ in \eqref{Hyp2}, we see 
that $m_\nu>M_\nu$, as a consequence of which we have $m_\nu\ge 1$. We write 
$m_1= m_\nu+r$. In view of \eqref{Hyp1}, we see that $r\ge 0$, and so an application of 
\eqref{HY} with $q=r+2$ leads to the inequality
$$ u_1^{r+1} u_\nu \le \frac{r+1}{r+2} u_1^{r+2} + \frac1{r+2} u_\nu^{r+2}. $$
Recall that $m_\nu\ge 1$, whence $m_1-r-1=m_\nu-1\ge 0$.  It follows that
$$ u_1^{m_1}u_\nu^{m_\nu}\le u_1^{m_1-r-1} u_\nu^{m_\nu-1}
\Big(  \frac{r+1}{r+2} u_1^{r+2} + \frac1{r+2} u_\nu^{r+2}\Big), $$
and thus
\begin{align*}
u_1^{m_1}\cdots u_\nu^{m_\nu} \le \frac{r+1}{r+2} u_1^{m_1+1} & 
u_2^{m_2}u_3^{m_3}\cdots u_{\nu-1}^{m_{\nu-1}} u_\nu^{m_\nu -1} \\ 
+&\frac1{r+2}u_1^{m_\nu -1}u_2^{m_2}u_3^{m_3}\cdots 
u_{\nu-1}^{m_{\nu-1}}u_\nu^{m_1+1}.
\end{align*}
The chain of exponents $m_1+1,m_2,m_3,\ldots,m_{\nu-1},m_\nu-1$ is decreasing, and we 
have $m_1+1\le M_1$ and $m_\nu-1\ge 0$. Hence, in view of \eqref{A5}, the hypotheses 
\eqref{Hyp1} and \eqref{Hyp2} are still met when we put $m_1+1$ in place of $m_1$ and 
$m_\nu-1$ in place of $m_\nu$. However, $m_1+1$ is closer to $M_1$ than is $m_1$, and 
likewise $m_\nu-1$ is closer to $M_\nu$ than is $m_\nu$. The value of $D$ associated with 
this new chain of exponents therefore decreases, and so we may apply the inductive 
hypothesis to find a weight $W$ on $S_\nu$ with
$$ u_1^{m_1+1}u_2^{m_2}u_3^{m_3}\cdots u_{\nu-1}^{m_{\nu-1}} u_\nu^{m_\nu -1} 
\le \sum_{\sigma\in S_\nu} W(\sigma) u_{\sigma(1)}^{M_1}u_{\sigma(2)}^{M_2}\cdots 
u_{\sigma(\nu)}^{M_\nu}. $$
Interchanging the roles of $u_1$ and $u_\nu$, and denoting by $\tau$ the transposition of 
$1$ and $\nu$, we obtain in like manner the bound
$$ u_\nu^{m_1+1}u_2^{m_2}u_3^{m_3}\cdots u_{\nu-1}^{m_{\nu-1}} u_1^{m_\nu -1} 
\le \sum_{\sigma\in S_\nu} W(\sigma\circ \tau)u_{\sigma(1)}^{M_1}u_{\sigma(2)}^{M_2}
\cdots u_{\sigma(\nu)}^{M_\nu}. $$
If we now import the last two inequalities into the inequality preceding them, we find that 
\eqref{A1} holds with 
$$ w(\sigma) = \frac{r+1}{r+2} W(\sigma) + \frac1{r+2}W(\sigma\circ \tau), $$
and $w$ is a weight on $S_\nu$. This completes the induction in the second case.
\end{proof}

\begin{lemma}\label{lemA2}
Suppose that $m_j$, $M_j$ $(1\le j\le \nu)$ satisfy \eqref{Hyp1} and \eqref{Hyp2}. For 
$1\le j\le \nu$ let $h_j:\mathbb R^n \to [0,\infty)$ denote a Lebesgue measurable function. 
Then
$$ \int h_{1}^{m_1}h_{2}^{m_2}\cdots h_{\nu}^{m_\nu}\, \mathrm d\mathbf{x}
\le \max_{\sigma\in S_\nu} \int  h_{\sigma(1)}^{M_1}h_{\sigma(2)}^{M_2}\cdots 
h_{\sigma(\nu)}^{M_\nu}\,\mathrm d\mathbf{x}. $$
\end{lemma}

\begin{proof}
Choose $u_j=h_j$ in Lemma \ref{lemA1} for $1\le j\le \nu$ and integrate.
\end{proof}

For applications to systems of diagonal equations or inequalities, functions $h_j$ come with 
an equivalence relation between them. This we encode as a partition of the set of indices 
$j$ in the final lemma of this section.

\begin{lemma}\label{lemA3}
Suppose that the exponents $m_j$, $M_j$ $(1\le j\le \nu)$ satisfy \eqref{Hyp1} and 
\eqref{Hyp2}. Let $s=m_1+m_2+\ldots +m_\nu$, and for $1\le j\le s$, let 
$h_j:\mathbb R^n\to [0,\infty)$ denote a Lebesgue measurable function. Finally, suppose 
that $J_1,J_2,\ldots ,J_\nu$ are sets with respective cardinalities $m_1,m_2,\ldots ,m_\nu$ 
that partition $\{1,2,\ldots,s\}$. Then, there exists a tuple $(i_1,\ldots ,i_\nu)$ and a 
permutation $\sigma\in S_\nu$, with $i_l\in J_{\sigma (l)}$ $(1\le l\le \nu)$, having the 
property that 
\begin{equation}\label{A7}
\int h_1h_2 \cdots h_s\,\mathrm d \mathbf{x} \le \int h_{i_1}^{M_1}h_{i_2}^{M_2}\ldots 
h_{i_\nu}^{M_\nu} \,\mathrm d\mathbf{x}. 
\end{equation}
\end{lemma}
 
\begin{proof} 
For each suffix $l$ with $1\le l\le \nu$, it follows from \eqref{T} that
$$ \prod_{j\in J_l} h_j \le \frac1{m_j}\sum_{j\in J_l} h_j^{m_j}. $$
Multiplying these inequalities together yields the bound
$$ h_1h_2 \cdots h_s \le \frac{1}{m_1\cdots m_\nu} \sum_{j_1\in J_1} \cdots
\sum_{j_\nu\in J_\nu} h_{j_1}^{m_1}h_{j_2}^{m_2} \cdots h_{j_\nu}^{m_\nu}. $$
Now integrate. One then finds that there exists a tuple $(j_1,\ldots ,j_\nu)$, with $j_l\in J_l$ 
$(1\le l\le \nu)$, for which
$$ \int h_1h_2 \cdots h_s\,\mathrm d \mathbf{x} \le \int h_{j_1}^{m_1}h_{j_2}^{m_2}
 \ldots h_{j_\nu}^{m_\nu} \,\mathrm d\mathbf{x}. $$
Finally, we apply Lemma \ref{lemA2}. One then finds that for some $\sigma\in S_\nu$ the 
upper bound \eqref{A7} holds with $i_l=j_{\sigma(l)}$ $(1\le l\le \nu)$.
\end{proof}
  
\section{Smooth Farey dissections}
In this section we describe a partition of unity that mimics the traditional Farey dissection. 
With other applications in mind, we work in some generality. Throughout this section we 
take $X$ and $Y$ to be real numbers with $1\le Y\le \frac12 \sqrt X$, and then let 
$\mathfrak N(q,a)$ denote the interval of all real $\alpha$ satisfying 
$|q\alpha -a |\le YX^{-1}$. Define $\mathfrak N=\mathfrak N_{X,Y} $ as the union of all 
$\mathfrak N(q,a)$ with $1\le q\le Y$, $a\in\mathbb Z$ and $(a,q)=1$. Note that the 
intervals $\mathfrak N(q,a)$ comprising $\mathfrak N$ are pairwise disjoint. We also write 
$\mathfrak M=\mathfrak M_{X,Y}$ for the set $\mathfrak N\cap [0,1]$. For appropriate 
choices of the parameter Y, the latter is a typical choice of major arcs in applications of the 
Hardy-Littlewood method.\par

The set $\mathfrak N$ has period 1. Its indicator function ${\bf 1}_{\mathfrak N}$ has 
finitely many discontinuities in $[0,1)$, implying unwanted delicacies concerning the 
convergence of the Fourier series of ${\bf 1}_{\mathfrak N}$. We avoid complications 
associated with this feature by a familiar convolution trick, which we now describe. 

Define the positive real number
$$ \kappa = \int_{-1}^1 \exp(1/(t^2-1))\,\d t, $$
and the function $K : \mathbb R\to [0,\infty)$ by
$$ K(t) = \left\{\begin{array}{ll} \kappa^{-1} \exp(1/(t^2-1)) & \text{if } |t|<1, \\
0 & \text{if } |t|\ge 1.\end{array} \right.
$$
As is well known, the function $K(t)$ is smooth and even. We scale this function with the 
positive parameter $X$ in the form
$$ K_X(t)= 4X\, K(4Xt). $$
Then $ K_X$ is supported on the interval $|t|\le 1/(4X)$ and satisfies the important relation
\begin{equation}\label{2.0}
\int_{-\infty}^\infty K_X(t)\d t = \int_{-\infty}^\infty  K(t)\d t=1.
\end{equation}
We now define the function ${\sf N}_{X,Y}: \mathbb R\to [0,1]$ by
\begin{equation}\label{2.1}
{\sf N}_{X,Y}(\alpha) = \int_{-\infty}^\infty {\bf 1}_{\mathfrak N}(\alpha-t) K_X(t)\d t = 
\int_{-\infty}^\infty {\bf 1}_{\mathfrak N}(t) K_X(\alpha-t)\d t.
\end{equation}
The main properties of this function ${\sf N}={\sf N}_{X,Y}$ are listed in the next lemma.

\begin{lemma}\label{majapprox}
The function ${\sf N}={\sf N}_{X,Y}$ is smooth, and for all $\alpha\in\mathbb R$ one has 
${\sf N}(\alpha)\in [0,1]$. Further, whenever $2\le Y\le \frac14 \sqrt X$, the inequalities 
\begin{equation}\label{2.2}
{\bf 1}_{{\mathfrak N}_{X, Y/2}} (\alpha) \le {\sf N} (\alpha)\le  
{\bf 1}_{{\mathfrak N}_{X,2 Y}}(\alpha) 
\end{equation}
and
\begin{equation}\label{2.3}
{\sf N}'(\alpha) \ll X, \quad  {\sf N}''(\alpha) \ll X^2
\end{equation}
hold uniformly in $\alpha\in\mathbb R$.
\end{lemma}
  
\begin{proof} The integrands in \eqref{2.1} are non-negative, so ${\sf N}(\alpha)\ge 0$, 
while \eqref{2.0} shows that ${\sf N}(\alpha)\le 1$. Since $K$ is smooth and compactly 
supported, the second integral formulation of ${\sf N}$ in \eqref{2.1} shows that 
$\sf N$ is smooth, and that the derivative is obtained by differentiating the integrand.  Thus, 
we obtain
$$ {\sf N}'(\alpha) = \int_{\mathfrak N} \frac{\partial}{\partial \alpha} K_X(\alpha-t)\d t, $$
whence
$$ |{\sf N}'(\alpha)| \le 4X \int_{-1}^1 |K'(t)|\d t. $$
This confirms the inequality 
for the first derivative in \eqref{2.3}. The bound for the second derivative follows in like manner by 
differentiating again.\par

We now turn to the task of establishing \eqref{2.2}. First suppose that 
$\alpha\in\mathfrak N_{X,Y/2}$. Then, there is a unique pair of integers $a\in\mathbb Z$ 
and $q\in\mathbb N$ with $(a,q)=1$, $q\le \frac12 Y$ and 
$|q\alpha-a|\le \frac12 YX^{-1}$. For $|t|\le (4X)^{-1}$ we then have
$$ \Big| (\alpha-t) - \frac{a}{q}\Big| \le\frac1{4X} + \frac{Y}{2qX} \le 
\frac{Y}{qX}. $$
Thus $\alpha-t \in\mathfrak N(q,a)\subseteq \mathfrak N_{X,Y}$. Since $K_X$ is supported 
on $[-1/(4X), 1/(4X)]$, we deduce from \eqref{2.0} and \eqref{2.1} that
$$ {\sf N}(\alpha) \ge \int_{-\infty}^\infty {\bf 1}_{\mathfrak N(q,a)}(\alpha-t)
K_X(t)\d t \ge \int_{-1}^1  K_X(t)\d t =1. $$
It follows that one has ${\sf N}(\alpha)=1$ for all $\alpha\in\mathfrak N_{X,Y/2}$. 
However, we know already that ${\sf N}(\alpha)$ is non-negative for all 
$\alpha\in\mathbb R$, and thus we have proved the first of the two inequalities in 
\eqref{2.2}.\par

We complete the proof of the lemma by addressing the second inequality in \eqref{2.2}. 
Suppose that ${\sf N}(\alpha)>0$. Then, it follows from \eqref{2.1} that for some 
$t\in \mathbb R$ with $|t|\le (4X)^{-1}$, one has 
$\alpha-t\in \mathfrak N_{X,Y}$. Hence, there exist $a\in\mathbb Z$ and $q\in\mathbb N$ 
with $(a,q)=1$, $q\le Y$ and $|\alpha-t-a/q|\le Y/(qX)$. By the triangle inequality, 
$$ \Big|\alpha-\frac{a}{q}\Big|\le \frac{Y}{qX} + \frac{1}{4X} \le \frac{2Y}{qX}.$$
This shows that $\alpha\in{\mathfrak N}_{X,2Y}$. Since $0\le {\sf N}(\alpha)\le 1$, the 
second of the inequalities in \eqref{2.2} also follows.
\end{proof}

We consider ${\sf N}= {\sf N}_{X,Y}$ as a smooth model of the major arcs 
$\mathfrak N_{X,Y}$. It is convenient to define corresponding minor arcs 
$\mathfrak n = \mathfrak n_{X,Y}$, with 
$\mathfrak n_{X,Y}= \mathbb R\setminus \mathfrak N_{X,Y}$, and to write 
$\mathfrak m = [0,1]\setminus \mathfrak M$ for the set of minor arcs complementary to 
$\mathfrak M$. The smoothed version of $\mathfrak n_{X,Y}$ is the function 
${\sf n}_{X,Y}:\mathbb R\to [0,1]$ defined by
$$ {\sf n}(\alpha) = \int_{-\infty}^\infty {\bf 1}_{\mathfrak n}(\alpha-t) K_X(t)\d t. $$
We trivially have ${\bf 1}_{\mathfrak N}(\alpha)+{\bf 1}_{\mathfrak n}(\alpha)=1$ for all 
$\alpha\in \mathbb R$, so it is a consequence of \eqref{2.0} and \eqref{2.1} that 
${\sf n}={\sf n}_{X,Y}$ satisfies the identity
\begin{equation}
\label{2.4} {\sf N}(\alpha)+{\sf n}(\alpha)=1.
\end{equation}
The properties of $\sf n$ can therefore be deduced from the corresponding facts 
concerning $\sf N$. In particular, Lemma \ref{majapprox} translates as follows.

\begin{lemma}\label{minapprox}
The function ${\sf n}={\sf n}_{X,Y}$ is smooth, and for all $\alpha\in\mathbb R$ one has 
${\sf n}(\alpha)\in [0,1]$. Further, whenever $2\le Y\le \frac14 \sqrt X$, the inequalities 
\[
{\bf 1}_{{\mathfrak n}_{X, 2Y}} (\alpha) \le {\sf n} (\alpha)\le  
{\bf 1}_{{\mathfrak n}_{X, Y/2}}(\alpha) 
\]
and
\[
{\sf n}'(\alpha) \ll X, \quad  {\sf n}''(\alpha) \ll X^2
\]
hold uniformly in $\alpha\in\mathbb R$.
\end{lemma}

\section{Fractional powers of Weyl sums}
In this section we consider a trigonometric polynomial
\begin{equation}
\label{3.1} T(\alpha) = \sum_{M<n\le M+N} c_n e(\alpha n)
\end{equation}  
with complex coefficients $c_n$. The associated ordinary polynomial
\begin{equation}
\label{3.2} P(z) = \sum_{n=1}^N c_{M+n} z^n
\end{equation} 
is related to $T$ via the identity
\begin{equation}
\label{TT}
T(\alpha) = e(M\alpha)P(e(\alpha)).
\end{equation}

\begin{lemma} \label{lem3.1}
Let $k\in\mathbb N$. Then, for any real number $u>k$, the real function 
$\Omega_u:\mathbb R \to \mathbb R$, defined by $\Omega_u(\alpha)=|T(\alpha)|^u$, is 
$k$ times continuously differentiable.
\end{lemma}

\begin{proof}
In view of \eqref{TT}, we see that it suffices to prove this result in the special case where 
$M=0$. This reduction step noted, we proceed by a succession of elementary exercises.\par

Let $u\in\mathbb R$. We begin by considering the function 
$\theta_u : \mathbb R\setminus\{0\} \to \mathbb R$ defined by 
$\theta_u(\alpha)=|\alpha|^u$. This function is differentiable on 
$ \mathbb R\setminus\{0\}$, and one has
$$ \theta'_u(\alpha) = u|\alpha|^u\alpha^{-1} =u\theta_u(\alpha)\alpha^{-1}.$$
By induction, it follows that for any $l\in\mathbb N$ the function $\theta_u$ is $l$ times 
differentiable, and that the $l$-th derivative is
\begin{equation}
\label{3.3} \theta_u^{(l)}(\alpha) = u(u-1)\cdots (u-l+1)\theta_u(\alpha)\alpha^{-l}. 
\end{equation} 
Now suppose that $u>0$. Then, by putting $\theta_u(0)=0$ we extend $\theta_u$ to a 
continuous function on $\mathbb R$. More generally, whenever $u>l$, then
$$ \lim_{\alpha\to 0} \frac{\theta_u(\alpha)}{\alpha^l}  = 0.$$
By \eqref{3.3}, this shows that whenever $u>l$ then $\theta_u^{(l)}$ extends to a 
continuous function on $\mathbb R$ by choosing $\theta_u^{(l)}(0)=0$, and that 
$\theta_u^{(l-1)}$ is differentiable at $0$ with derivative $0$. We summarize this last 
statement as follows:
\\[2ex]
(a) {\em Let $k\in\mathbb N$ and $u>k$. Then $\theta_u$ is $k$ times continuously 
differentiable on $\mathbb R$.}
\\

Next, for $u>0$, consider the function $\rho_u: \mathbb R\to \mathbb R$ defined by 
putting $\rho_u(\alpha)=|\sin \pi\alpha|^u$. For $\alpha\in (0,1)$ one has 
$\sin \pi \alpha >0$, whence $\rho_u(\alpha)=(\sin \pi \alpha)^u$. Thus $\rho_u$ is smooth 
on $(0,1)$. But $\rho$ has period $1$, so it suffices to examine its differentiability 
properties at $\alpha=0$, a point at which $\rho_u$ is continuous. For all real $\alpha$ we 
have $\sin \pi\alpha= \pi\alpha E(\alpha)$, where
$$ E(\alpha)= \sum_{j=0}^\infty (-1)^j \frac{(\pi\alpha)^{2j}}{(2j+1)!}. $$
The function $E$ is smooth on $\mathbb R$ with $E(0)=1$. Hence $E(\alpha)>0$ in a 
neighbourhood of $0$ where we then also have
$$ \rho_u(\alpha) = \pi^u |\alpha|^u E(\alpha)^u. $$
By applying the product rule in combination with our earlier conclusion (a), we therefore 
conclude as follows:
\\[2ex]
(b) {\em Let $k\in\mathbb N$ and $u>k$. Then $\rho_u$ is $k$ times continuously 
differentiable on} $\mathbb R$.
\\

We now turn to the function $T$ where we suppose that $M=0$, as we may. The sum in 
\eqref{3.1} defines a holomorphic function of the complex variable $\alpha$, and hence the 
function $T:\mathbb R\to\mathbb C$ is a smooth map of period $1$. The sum
$$ \bar T(\alpha)= \sum_{1\le n\le N} \bar c_n e(-\alpha n) $$
defines another trigonometric polynomial, and for $\alpha\in\mathbb R$ we have 
$\overline{T(\alpha)}= \bar T(\alpha)$. Consequently, for real $\alpha$ we have
\begin{equation}
\label{3.31} |T(\alpha)|^2 = T(\alpha)\bar T(\alpha),
\end{equation} 
whence the function $|T|^2:\mathbb R\rightarrow \mathbb C$, given by 
$\alpha\mapsto |T(\alpha)|^2$, is smooth on $\mathbb R$ with
\begin{equation}
\label{3.4} \frac{\mathrm d}{\mathrm d \alpha}\,|T(\alpha)|^2 = 
T'(\alpha)\bar T(\alpha) + T(\alpha)\bar T'(\alpha). 
\end{equation}
On noting that $T(\alpha)^j$ is again a trigonometric polynomial for all $j\in \mathbb N$, 
we see that $|T(\alpha)|^{2j}$ is smooth. Hence, from now on, we may suppose that $u$ 
{\em is a real number but not an even natural number}. Also, the conclusion of Lemma 
\ref{lem3.1} is certainly true in the trivial case where  $c_n=0$ for all $n$. In the contrary 
case, the polynomial in \eqref{3.2} has at most finitely many zeros. Therefore, the set
$$ Z= \{\alpha\in\mathbb R: T(\alpha)=0\} $$
is $1$-periodic with $Z\cap [0,1)$ finite, and consequently $\mathbb R\setminus Z$ is open.

\par We next examine the function $|T|^u:\mathbb R\setminus Z\rightarrow \mathbb C$, given by 
$\alpha\mapsto |T(\alpha)|^u$.  
\\[2ex]
(c) {\em When $u$ is real but not an even natural number, the function 
$|T|^u$ is smooth}. 
\\[2ex]
In order to confirm this assertion, note that $|T(\alpha)|^u=\theta_{u/2}(|T(\alpha)|^2)$. 
By applying the chain rule in combination with the preamble to conclusion (a) and 
\eqref{3.4}, we find that $|T(\alpha)|^u$ is differentiable for 
$\alpha\in\mathbb R \setminus Z$. Indeed,
\begin{align}
\frac{\mathrm d}{\mathrm d \alpha}\,|T(\alpha)|^u &=\theta'_{u/2}(|T(\alpha)|^2)
\big(T'(\alpha)\bar T(\alpha) + T(\alpha)\bar T'(\alpha)) \notag\\
& = \frac{u}{2}|T(\alpha)|^{u-2}  \big(T'(\alpha)\bar T(\alpha) + T(\alpha)\bar T'(\alpha)).
\label{3.41}
\end{align} 
Since the final factor on the right hand side here is smooth, we may repeatedly apply the 
product rule to conclude that $|T(\alpha)|^u$ is smooth on $\mathbb R\setminus Z$, as 
claimed.\medskip

Finally, we consider any element $\alpha_0\in Z$. Then one has $P(e(\alpha_0))=0$. Since 
$P$ is not the zero polynomial, there exists $r\in\mathbb N$ and a polynomial 
$Q\in \mathbb C[z]$ with $Q(e(\alpha_0))\neq 0$ such that $P(z)=(z-e(\alpha_0))^r Q(z)$. 
Write $U(\alpha)=Q(e(\alpha))$ for the trigonometric polynomial associated with $Q$. Then 
$T(\alpha)=\big(e(\alpha)-e(\alpha_0)\big)^r U(\alpha)$. For $u>0$ and all real $\alpha$ 
we then have
$$ |T(\alpha)|^u=|e(\alpha)-e(\alpha_0)|^{ru} |U(\alpha)|^u 
= |2\sin \pi(\alpha-\alpha_0)|^{ru} |U(\alpha)|^u. $$ 
There is an open neighbourhood of $\alpha_0$ on which $U(\alpha)$ does not vanish. By 
our conclusion (c) it is apparent that $|U(\alpha)|^u$ is smooth on this neighbourhood. If 
$u>k$, then the conclusion (b) implies that the function $|2\sin \pi(\alpha-\alpha_0)|^{ru}$ 
is $k$ times continuously differentiable. The conclusion of the lemma therefore follows by 
application of the product rule.
\end{proof}

We mention in passing that if more is known about the zeros of $P$, then the argument that 
we have presented shows more. For example, if all the zeros in $Z$ are double zeros and 
$u>k$, then $|T(\alpha)|^u$ is $2k$ times differentiable.

\begin{lemma}\label{lem3.2}
Let $W:\mathbb R \to \mathbb R$ be a twice continuously differentiable function of period 
$1$, and let $u\ge 2$. For $l\in\mathbb Z$ let
\begin{equation}\label{3.51}
b_l = \int_0^1 W(\alpha) |T(\alpha)|^u e(-\alpha l)\d\alpha.
\end{equation}
Then, for all $l\in\mathbb Z\setminus\{0\}$, one has
\begin{equation}
\label{3.5} |b_l|\le \frac{1}{(2\pi l)^2} \int_0^1 \Big| 
\frac{\mathrm d^2}{\mathrm d \alpha^2}\,W(\alpha) |T(\alpha)|^u\Big|\d\alpha.
\end{equation}
Moreover, for all $\alpha\in\mathbb R$ one has the Fourier series expansion
\begin{equation}\label{3.6}
  W(\alpha)|T(\alpha)|^u = \sum_{l\in \mathbb Z}b_le(\alpha l),
\end{equation}
in which the right hand side converges absolutely and uniformly on $\mathbb R$.
\end{lemma}

\begin{proof}
By \eqref{3.31} and Lemma \ref{lem3.1}, the condition $u\ge 2$ ensures that 
$W(\alpha)|T(\alpha)|^u$ is twice continuously differentiable. Hence, the integral on the 
right hand side of \eqref{3.5} exists, and the upper bound \eqref{3.5} follows from 
\eqref{3.51} by integrating by parts two times. Furthermore, the upper bound \eqref{3.5} 
ensures that the series in \eqref{3.6} converges absolutely and uniformly on $\mathbb R$. 
Thus, by \cite[Chapter II, Theorem 8.14]{Z}, this Fourier series sums to 
$W(\alpha)|T(\alpha)|^u$.
\end{proof}

In this paper Lemmata \ref{lem3.1} and \ref{lem3.2} will only be used with the quartic Weyl 
sum $f$, as defined in \eqref{ff}, in the role of $T$. The weight $W$ will be either 
constantly $1$ or a smooth minor arc. Let $u>0$ and define the Fourier coefficient
\begin{equation}\label{3.8}
\psi_u(n) = \int_0^1 |f(\alpha)|^u e(-\alpha n) \d\alpha.
\end{equation} 
Also, with a parameter $Y$ at our disposal within the range $1\le Y\le \frac14 P^2$, we 
consider the smooth minor arcs ${\sf n}(\alpha )={\sf n}_{P^4,Y}(\alpha)$ and introduce 
the related Fourier coefficient 
\begin{equation}\label{3.9}
\phi_u(n) = \int_0^1 {\sf n}(\alpha)|f(\alpha)|^u e(-\alpha n) \d\alpha.
\end{equation}

\begin{lemma} \label{lem3.3}
Suppose that $u\ge 2$ and $1\le Y\le \tfrac{1}{4}P^2$. Then, for all 
$n\in\mathbb Z\setminus\{0\}$, one has
$$ |\phi_u(n)|+|\psi_u(n)|\ll P^{u+8}n^{-2}. $$
\end{lemma}

\begin{proof}
We first compute the derivatives of $|f(\alpha)|^u$. Suppose temporarily that $u$ is not an 
even natural number. By \eqref{3.41}, whenever $f(\alpha)\neq 0$, we have
$$ \frac{\mathrm d}{\mathrm d \alpha} \, |f(\alpha)|^u =
\frac{u}{2}\, |f(\alpha)|^{u-2}\left( f'(\alpha)\bar f(\alpha)+f(\alpha)\bar f'(\alpha)\right),$$
and we may differentiate again to confirm the identity
\begin{align*}
\frac{\mathrm d^2}{\mathrm d \alpha^2} \, |f(\alpha)|^u=&\frac{u(u-2)}4\,
|f(\alpha)|^{u-4}\big(f'(\alpha)\bar f(\alpha)+f(\alpha)\bar f'(\alpha)\big)^2\\ 
& + \frac{u}{2}\, |f(\alpha)|^{u-2} \big(f''(\alpha)\bar f(\alpha)
+2 f'(\alpha)\bar f'(\alpha) + f(\alpha)\bar f''(\alpha)\big).
\end{align*}
These formulae hold for all $\alpha \in \mathbb R$ when $u$ is an even natural number, and 
thus
$$ \Big| \frac{\mathrm d}{\mathrm d \alpha} \, |f(\alpha)|^u \Big|
\le u |f(\alpha)|^{u-1} |f'(\alpha)| $$
and
$$ \Big|\frac{\mathrm d^2}{\mathrm d \alpha^2} \, |f(\alpha)|^u \Big|
\le u(u-1) |f(\alpha)|^{u-2} |f'(\alpha)|^2 + u |f(\alpha)|^{u-1} |f''(\alpha)|.$$
Hence, the trivial estimates $f(\alpha)\ll P$, $f'(\alpha)\ll P^5$ and $f''(\alpha)\ll P^9$ 
suffice to conclude that the upper bounds
\begin{equation}\label{3.10}
\frac{\mathrm d}{\mathrm d \alpha} \, |f(\alpha)|^u \ll P^{u+4}\quad \text{and}\quad 
\frac{\mathrm d^2}{\mathrm d \alpha^2} \, |f(\alpha)|^u \ll P^{u+8}
\end{equation}  
hold for all $\alpha\in\mathbb R$ when either $u=2$ or $f(\alpha)\neq 0$. However, when 
$u>2$ these derivatives will be zero whenever $f(\alpha)=0$, so the inequalities 
\eqref{3.10} hold uniformly in $\alpha\in\mathbb R$. The upper bound 
$\psi_u(n)\ll P^{u+8}n^{-2}$ is now immediate from Lemma \ref{lem3.2}. Furthermore,  an application of the product rule in combination with Lemma 
\ref{minapprox} and \eqref{3.10} shows that
$$\frac{\mathrm d}{\mathrm d \alpha} \,{\sf n}(\alpha) |f(\alpha)|^u \ll P^{u+4}\quad 
\text{and}\quad \frac{\mathrm d^2}{\mathrm d \alpha^2} \, {\sf n}(\alpha)|f(\alpha)|^u 
\ll P^{u+8}.$$
The estimate $\phi_u(n)\ll P^{u+8}n^{-2}$ therefore follows by invoking Lemma 
\ref{lem3.2} once again, and this completes the proof of the lemma.
\end{proof}

\section{Cubic moments of Fourier coefficients}
The principal results in this section are the upper bounds for cubic moments of $\phi_u(n)$ 
and $\psi_u(n)$ embodied in Theorem \ref{thm4.1} below. The proof of these estimates 
involves a development of the ideas underpinning the main line of thought in our earlier 
paper \cite{Jems1}. For $u>0$ it is convenient to define
\begin{equation}\label{4.1}
\delta(u)=(25-3u)/6.
\end{equation}
In many of the computations later it is useful to note that
\begin{equation}\label{4.13}
3u-8+\delta(u)= \frac52 u -\frac{23}6.
\end{equation}

\begin{theorem}\label{thm4.1}
Let $u$ be a real number with $6\le u\le 25/3$. Then
\begin{equation}
\label{4.2} \sum_{n\in \mathbb Z}|\psi_u(n)|^3 \ll P^{3u-8+\delta(u)+\varepsilon}.
\end{equation}
Further, when $2P^{4/15}\le Y\le P/16$ and $6\le u\le 11$, one has
\begin{equation}\label{4.3}
\sum_{n\in \mathbb Z}|\phi_u(n)|^3 \ll P^{3u-8+\delta(u)+\varepsilon}.
\end{equation}
\end{theorem}

When $u\ge 6$, the contribution from the major arcs to the sum in \eqref{4.2} is easily seen 
to be of order $P^{3u-8}$. Since $\delta(u)$ is negative for $u>25/3$, we cannot 
expect that the upper bound \eqref{4.2} holds for such $u$. However, as is evident from 
\eqref{4.3}, a minor arcs version remains valid for $u\le 11$. Before we embark on the 
proof of this theorem, we summarize some mean value estimates related to the Weyl sum 
\eqref{ff}. In the following two lemmata, we assume that $1\le Y\le P/8$ and write 
$\mathfrak M=\mathfrak M_{P^4,Y}$ and $\mathfrak m=\mathfrak m_{P^4,Y}$. It is 
useful to note that $\mathfrak m_{P^4,Y}=\mathfrak m_{P^4,P/8}\cup\mathfrak K$, 
where $\mathfrak K = \mathfrak M_{P^4,P/8} \setminus \mathfrak M_{P^4,Y}$. Then, 
from \cite[Lemma 5.1]{V89}, we have the bounds
\begin{equation}\label{4.51}
\int_{\mathfrak M} |f(\alpha)|^{6}\d \alpha \ll P^{2}\quad \text{and}\quad 
\int_{\mathfrak K} |f(\alpha)|^{6}\d \alpha \ll P^{2}Y^{\varepsilon-1/4}. 
\end{equation}

\begin{lemma}\label{lem4.1}
Suppose that $P^{4/15}\le Y\le P/8$. Then
$$ \int_{\mathfrak m} |f(\alpha)|^{20}\d \alpha \ll P^{15+\varepsilon}. $$
\end{lemma}
\begin{proof}
For $Y=P/8$, the desired estimate is the case $k=4$, $w=20$ of Wooley 
\cite[Lemma 3.1]{Woo2016}. For smaller values of $Y$, we make use of the case 
$Y=P/8$ and apply the second bound of \eqref{4.51}. On combining 
\cite[Theorem 4.1]{hlm} with \cite[Lemma 2.8 and Theorem 4.2]{hlm}, moreover, one 
readily confirms that the upper bound $f(\alpha)\ll PY^{-1/4}$ holds uniformly  for 
$\alpha\in\mathfrak K$. Consequently, one has the estimate
$$ \int_{\mathfrak K} |f(\alpha)|^{20}\d \alpha \ll P^{16}Y^{\varepsilon-15/4},$$
and the conclusion of the lemma follows. 
\end{proof}

\begin{lemma}\label{lemma5.3}
When $8\le u\le 14$, one has
\begin{equation}\label{4.4}  
\int_0^1 |f(\alpha)|^u\d \alpha \ll P^{\frac56 u-\frac53+\varepsilon}. 
\end{equation}
Meanwhile, when $8\le u\le 20$, then uniformly in $P^{4/15}\le Y\le P/8$, one has
\begin{equation}\label{4.5}  
\int_{\mathfrak m} |f(\alpha)|^u\d \alpha \ll P^{\frac56 u-\frac53+\varepsilon}. 
\end{equation}
\end{lemma}
 
\begin{proof}
It is a consequence of Hua's Lemma \cite[Lemma 2.5]{hlm} that
\begin{equation}\label{4.6}
\int_\grm |f(\alpha)|^8\d \alpha \le \int_0^1 |f(\alpha)|^8\d \alpha \ll P^{5+\varepsilon}.
\end{equation}
One interpolates linearly between this estimate and the bound established in Lemma 
\ref{lem4.1} via H\"older's inequality to confirm the upper bound \eqref{4.5} for 
$8\le u\le 20$. The upper bound \eqref{4.4} then follows on noting that for $6\le u\le 14$, it 
follows from \eqref{4.51} that
\[
\int_{\mathfrak M}|f(\alpha)|^u\d \alpha \ll P^{u-4}\ll P^{\frac56 u-\frac53}.
\]
Since $[0,1]=\mathfrak M\cup \mathfrak m$, the desired conclusion follows at once.
\end{proof}

In the special case $u=14$, the first conclusion of Lemma \ref{lemma5.3} assumes the 
simple form already announced in (\ref{1.5}). 

\begin{lemma}\label{lem4.3}
Let $\mathscr Z$ be a set of $Z$ integers. Then 
$$ \int_0^1 \bigg|\sum_{z\in\mathscr Z} e(\alpha z) \bigg|^2 |f(\alpha)|^2\d\alpha 
\ll PZ +P^{1/2+\varepsilon}Z^{3/2} $$
and
$$ \int_0^1 \Big|\sum_{z\in\mathscr Z} e(\alpha z) \Big|^2 |f(\alpha)|^4\d\alpha 
\ll P^3 Z +P^{2+\varepsilon}Z^{3/2}. $$
\end{lemma}
\begin{proof}
This is essentially contained in \cite[Lemma 6.1]{KW}, where these estimates are 
established in the case when $\mathscr Z$ is contained in $[0,P^4]$. As pointed out in 
\cite[Lemma 2.2]{BW21} this condition is not required.
\end{proof}

We now have available sufficient infrastructure to derive upper bounds for cubic moments 
of $\phi_u(n)$ and $\psi_u(n)$.

\begin{proof}[The proof of Theorem \ref{thm4.1}] Let $\vartheta_u (n)$ denote one of 
$\psi_u(n)$, $\phi_u(n)$. On examining the statement of the theorem, it is apparent that we 
may assume that in the former case we have $6\le u\le 25/3$, and in the latter case 
$6\le u\le 11$ and $2P^{4/15}\le Y\le P/16$. We begin with the observation that, by Lemma 
\ref{lem3.3}, one has $\vartheta_u (n) \ll P^{u+8} n^{-2}$. Consequently, when $u\ge 6$, 
one has
$$\sum_{|n|>P^7}|\vartheta_u(n)|^3+
\sum_{\substack{|n|\le P^7\\ |\vartheta_u(n)|\le 1}}|\vartheta_u(n)|^3
\ll P^7+P^{3u+24}\sum_{|n|>P^7}n^{-6}\ll P^{3u-11}.$$

\par It remains to consider the contribution of those integers $n$ with $|n|\le P^7$ and 
$|\vartheta_u(n)|>1$. We put $\Theta(\alpha)=1$ when $\vartheta_u=\psi_u$, and 
$\Theta(\alpha)={\sf n}(\alpha)$ when $\vartheta_u=\phi_u$. Then the definitions 
\eqref{3.8} and \eqref{3.9} take the common form
\begin{equation}\label{Thet}
\vartheta_u(n) = \int_0^1 \Theta(\alpha)|f(\alpha)|^u e(-\alpha n)\d\alpha.
\end{equation}
By Lemma \ref{minapprox}, it follows that $\Theta(\alpha)\in [0,1]$. Thus, by Lemma 
\ref{lemma5.3}, one finds that
$$|\vartheta_u(n)|\le \vartheta_u(0)\le \psi_u(0)\ll P^{\frac56 u-\frac53+\varepsilon}\quad 
(8\le u\le 11).$$
In the missing cases where $6\le u< 8$ one interpolates between \eqref{4.6} and the 
elementary inequality
\begin{equation}\label{Hua4}
\int_0^1|f(\alpha)|^4\d\alpha \ll P^{2+\varepsilon},
\end{equation}
also a consequence of Hua's Lemma \cite[Lemma 2.5]{hlm}, to conclude that
$$|\vartheta_u(n)|\le \vartheta_u(0)\le \psi_u(0)\ll P^{2+\frac34(u-4)+\varepsilon}. $$

Fix a number $\tau$ with $0<\tau< 10^{-10}$ and define $T_0$ by
$$T_0=\begin{cases} P^{\frac34 u-1+\tau},&\text{when $6\le u<8$},\\
P^{\frac56 u-\frac53+\tau},&\text{when $8\le u\le 11$}.\end{cases}$$
Then, on recalling the upper bounds for $\vartheta_u(n)$ just derived,  a 
familiar dyadic dissection argument shows that there is a number $T\in [1,T_0]$ with the 
property that
\begin{align} \sum_{n\in \mathbb Z}|\vartheta_u(n)|^3 &\ll P^{3u-11} + 
(\log P)\sum_{\substack{|n|\le P^7\\ T<|\vartheta_u(n)|\le 2T}}|\vartheta_u(n)|^3\notag\\
&\ll P^{3u-11} + P^\varepsilon T^3 Z,\label{4.7}
\end{align}
where $Z$ denotes the number of elements in the set
$$ {\mathscr Z}=\{n\in\mathbb Z: \text{$|n|\le P^7$ and $T<|\vartheta_u(n)|\le 2T$}\}.$$
For each $n\in\mathscr Z$ there is a complex number $\eta_n$, with $|\eta_n|=1$, for 
which $\eta_n \vartheta_u(n)$ is a positive real number. Write
\begin{equation}
\label{4.8} K(\alpha)=\sum_{n\in\mathscr Z} \eta_n e(-\alpha n). 
\end{equation} 
Then one concludes from \eqref{Thet} and orthogonality that
\begin{equation}\label{4.9}
TZ<\sum_{n\in \mathscr Z}\eta_n\vartheta_u(n)=
\int_0^1 \Theta(\alpha)K(\alpha)|f(\alpha)|^u \d\alpha.
\end{equation}

Beyond this point our argument depends on the size of $T$. Our first argument handles 
the small values $T\le P^{\frac56u-\frac{35}{18}}$. By \eqref{4.9} and H\"older's 
inequality, we obtain the bound
\begin{equation}\label{4.11}
TZ \le I^{1/2}\biggl(\int_0^1 |K(\alpha)^2f(\alpha)^4|\d\alpha\biggr)^{1/3}
\biggl(  \int_0^1 |K(\alp)|^2\d\alp \biggr)^{1/6},
\end{equation}
where
$$I=\int_0^1 \Theta(\alpha)^2|f(\alp)|^{2u-\frac83}\d\alp .$$
By orthogonality, one has
$$ \int_0^1 |K(\alp)|^2\d\alp = Z, $$
and by a consideration of the underlying Diophantine equations, one deduces via Lemma 
\ref{lem4.3} that
\begin{equation}\label{4.10}
\int_0^1 |K(\alpha)^2f(\alpha)^4|\d\alpha \ll P^3Z + P^{2+\varepsilon}Z^{3/2}. 
\end{equation}
Next we confirm the bound $I\ll P^{\frac53u-\frac{35}9+\varepsilon}$. Indeed, in the case 
where $\Theta=1$ we have $6\le u\le 25/3$. In such circumstances $8<2u-8/3\le 14$, and 
so \eqref{4.4} applies and yields the claimed bound. In the case $\Theta={\sf n}$ we have 
$u\le 11$, and hence $2u-8/3<20$. Write $\mathfrak m=\mathfrak m_{P^4,Y/2}$. Then by 
Lemma \ref{minapprox}, we have $0\le {\sf n}(\alpha)\le {\bf 1}_{\mathfrak m}$. We 
therefore deduce that in this second case we have
$$I\le \int_0^1 {\sf n}(\alpha)|f(\alpha)|^{2u-\frac{8}{3}}\d\alpha \le \int_\grm 
|f(\alpha)|^{2u-\frac{8}{3}}\d\alpha ,$$
and \eqref{4.5} confirms our claimed bound for $I$.\par

Collecting these estimates together within \eqref{4.11}, we now have
\[
TZ \ll P^\varepsilon \left(P^3Z + P^2Z^{3/2}\right)^{1/3}Z^{1/6}
\bigl(P^{\frac53 u-\frac{35}9}\bigr)^{1/2}.
\]
On recalling \eqref{4.13}, we find that this relation disentangles to yield the bound
\begin{align*}
T^3Z &\ll P^{2+\frac32(\frac53 u-\frac{35}9)+\varepsilon} 
+TP^{2+\frac53 u-\frac{35}9+\varepsilon}\\
& = P^{3u-8 +\delta(u)+\varepsilon} + TP^{\frac53 u -\frac{17}9+\varepsilon}.
\end{align*}
It transpires that in the range $T\le P^{\frac56u-\frac{35}{18}}$ the first term on the right 
hand side dominates, so that we finally reach the desired conclusion 
$T^3Z \ll P^{3u-8 +\delta(u)+\varepsilon}$. In view of \eqref{4.7}, this is enough to 
complete the proof of Theorem \ref{thm4.1} in the case that $T$ is small.\par

Our second approach is suitable for $T$ of medium size, with
\begin{equation}\label{4.14}
P^{\frac56u-\frac{35}{18}}<T\le P^{\frac56 u - \frac{11}{6}}.
\end{equation}
We apply Schwarz's inequality to \eqref{4.9}, obtaining the bound
$$ TZ\le \biggl(\int_0^1 |K(\alpha)^2f(\alpha)^4|\d\alpha\biggr)^{1/2}
\biggl( \int_0^1 \Theta(\alpha)^2|f(\alp)|^{2u-4}\d\alp\biggr)^{1/2}. $$
Note that when $6\le u\le 11$, one has $8\le 2u-4\le 18$, and when instead $u\le 25/3$, 
we have $2u-4<14$. Hence, as in the proof of our earlier estimate for $I$, it follows from 
Lemma \ref{lemma5.3} that
$$ \int_0^1 \Theta(\alpha)^2|f(\alp)|^{2u-4}\d\alp \ll P^{\frac53 u - 5+\varepsilon}. $$
Applying this estimate in combination with \eqref{4.10}, we conclude that
$$TZ \ll  P^\varepsilon (P^3Z+P^2Z^{3/2})^{1/2}(P^{\frac53 u-5})^{1/2}. $$
This bound disentangles to deliver the relation
$$ T^3Z \ll TP^{\frac53 u-2+\varepsilon}+ T^{-1} P^{\frac{10}{3}u-6+\varepsilon}.
$$
On recalling \eqref{4.13}, we find that our present assumptions \eqref{4.14} concerning 
the size of $T$ deliver the estimate
$$T^3Z \ll P^{\frac52 u-\frac{23}{6}+\varepsilon}+
P^{\frac{5}{2}u-\frac{73}{18}+\varepsilon}\ll P^{3u-8+\delta(u)+\varepsilon}.$$
The conclusion of Theorem \ref{thm4.1} again follows in this case, by virtue of \eqref{4.7}.

\par The analysis of the large values $T$ satisfying $P^{\frac56 u-\frac{11}{6}}<T\le T_0$ 
is more subtle. Suppose temporarily that $\vartheta_u=\psi_u$, and hence that $u\le 25/3$. 
Then, by \eqref{2.4} and \eqref{4.9},
$$ TZ  \le \int_0^1 {\sf N}(\alpha)K(\alpha)|f(\alpha)|^u \d\alpha
+ \int_0^1 {\sf n}(\alpha)K(\alpha)|f(\alpha)|^u \d\alpha.$$
By hypothesis, we have $u\ge 6$. Also, from Lemma \ref{majapprox}, we have
${\sf N}\le {\bf 1}_{{\mathfrak N}_{P^4,P/8}}$, so that \eqref{4.51} yields the bound
$$ \int_0^1 {\sf N}(\alpha)K(\alpha)|f(\alpha)|^u \d\alpha
\le Z \int_{{\mathfrak M}_{P^4,P/8}} |f(\alpha)|^u \d\alpha \ll ZP^{u-4}.$$
Since $u-4<\frac56 u - \frac{11}{6}$, for large enough $P$ one has 
$ZP^{u-4}<\frac 12 TZ$. Thus
\begin{equation}\label{4.16}
TZ \ll \int_0^1 {\sf n}(\alpha)K(\alpha)|f(\alpha)|^u \d\alpha.
\end{equation}
Note that this is exactly the inequality \eqref{4.9} in the case where $\vartheta_u=\phi_u$. 
Consequently, the upper bound \eqref{4.16} holds for the large values of $T$ currently 
under consideration, irrespective of the choice of $\vartheta_u$. Now apply Schwarz's 
inequality to \eqref{4.16}. Then, by Lemma \ref{minapprox}, we 
deduce that
$$ TZ \le \biggl(\int_0^1 |K(\alpha)f(\alpha)|^2 \d\alpha\biggr)^{1/2}
\biggl( \int_{\mathfrak m}|f(\alp)|^{2u-2}\d\alp\biggr)^{1/2}, $$
where again we write $\mathfrak m=\mathfrak m_{P^4,Y/2}$. Note here that $u\le 11$, 
so that $2u-2\le 20$. Hence, by Lemmata \ref{lemma5.3} and \ref{lem4.3}, we have
$$ TZ \ll P^\varepsilon \bigl( PZ+P^{\frac12}Z^{\frac32}\bigr)^{1/2}
\bigl( P^{\frac53 u-\frac{10}3}\bigr)^{1/2}.$$
Consequently, our assumptions concerning the size of $T$ reveal that
\begin{align}
T^3Z&\ll TP^{\frac53 u-\frac{7}{3}+\varepsilon} 
+T^{-1}P^{\frac{10}3 u-\frac{17}3+\varepsilon}\notag \\
&\ll T_0P^{\frac53 u-\frac{7}{3}+\varepsilon}+P^{\frac52 u-\frac{23}{6}+\varepsilon}.
\label{5.z}
\end{align}
When $6\le u<8$, one has
$$\bigl(\tfrac{3}{4}u-1\bigr)+\bigl(\tfrac{5}{3}u-\tfrac{7}{3}\bigr) 
=\tfrac{29}{12}u-\tfrac{10}{3}\le \tfrac{5}{2}u-\tfrac{23}{6},$$
whilst for $8\le u\le 11$,
$$\bigl(\tfrac{5}{6}u-\tfrac{5}{3}\bigr)+\bigl(\tfrac{5}{3}u-\tfrac{7}{3}\bigr) 
=\tfrac{5}{2}u-4<\tfrac{5}{2}u-\tfrac{23}{6}.$$
Then in either case one finds from \eqref{5.z} via \eqref{4.13} that 
$T^3Z\ll P^{3u-8+\delta(u)+2\tau}$, and the conclusion of Theorem \ref{thm4.1} 
follows in this final case, again by \eqref{4.7}, on taking $\tau$ sufficiently small.
\end{proof}

We close this section with a related but simpler result.

\begin{theorem}\label{thm4.4}
One has
$$ \sum_{n\in \mathbb Z}\psi_4(n)^3 \ll P^{13/2+\varepsilon}.$$
\end{theorem} 

\begin{proof} By \eqref{3.8} and orthogonality, the Fourier coefficient $\psi_4(n)$ has a 
Diophantine interpretation that shows on the one hand that $\psi_4(n)\in\mathbb N_0$, and 
on the other that $\psi_4(n)=0$ for all $n\in\mathbb Z$ with $|n|>2P^4$. By \eqref{3.8} 
and \eqref{Hua4}, we also have the bound $ \psi_4(n)\le \psi_4(0) \ll P^{2+\eps}$. The 
argument leading to \eqref{4.7} now shows that there is a number $T$ with 
$1\le T \le P^{2+\varepsilon}$ having the property that
\begin{align}
\sum_{n\in \mathbb Z}\psi_4(n)^3 &\ll P^{6+\varepsilon}+
P^\varepsilon\sum_{\substack{|n|\le 2P^4\\ T\le \psi_4(n)\le 2T}} \psi_4(n)^3\notag \\
&\ll P^{6+\varepsilon}+P^\varepsilon T^3Z,\label{5.w}
\end{align}
where $Z$ denotes the number of elements in the set
$$ {\mathscr Z}=\{n\in\mathbb Z: \text{$|n|\le 2P^4$ and $T<|\psi_4(n)|\le 2T$}\}.$$
As in the corresponding analysis within the proof of Theorem \ref{thm4.1}, we next find 
that there are unimodular complex numbers $\eta_n$ $(n\in\mathscr Z)$ having the 
property that, with $K(\alpha)$ defined via \eqref{4.8}, one has
$$TZ <\int_0^1 K(\alpha)|f(\alpha)|^4 \d\alpha.$$

We first handle small values of $T$. Here, an application of Schwarz's inequality leads via 
\eqref{4.6} to the bound
$$TZ\le \biggl(\int_0^1 |f(\alpha)|^8\d\alpha\biggr)^{1/2}
\biggl(  \int_0^1 |K(\alp)|^2\d\alp \biggr)^{1/2}\ll P^{5/2+\eps}Z^{1/2}.$$
This disentangles to yield $T^3Z \ll TP^{5+\varepsilon}$, proving the theorem for 
$T\le P^{3/2}$.\par

Next, when $T$ is large, we apply H\"older's inequality in a manner similar to that employed 
in the large values analysis of the proof of Theorem \ref{thm4.1}. Thus
$$ TZ \le \biggl(\int_0^1 |K(\alpha)^2f(\alpha)^2|\d\alpha\biggr)^{1/2}
\biggl( \int_0^1 |f(\alp)|^4\d\alp \biggr)^{1/4} \biggl( \int_0^1 |f(\alp)|^8\d\alp 
\biggr)^{1/4},$$
and hence
$$ TZ\ll P^\varepsilon (PZ+ P^{1/2}Z^{3/2})^{1/2}P^{7/4}. $$
We now obtain the bound
$$ T^3Z \ll TP^{9/2+\varepsilon} +T^{-1}P^{8+\varepsilon},$$
and in view of \eqref{5.w}, this proves Theorem \ref{thm4.4} in the complementary case 
$P^{3/2}\le T\le P^{2+\varepsilon}$.
\end{proof}
 
\section{Mean values of quartic Weyl sums}
In this section we estimate certain entangled moments of quartic Weyl sums, and then apply 
them to obtain minor arc estimates for use within the proofs of Theorems \ref{theorem1.1} 
and \ref{theorem1.2}. Throughout this section and the next, let the pair of integers 
$c_i, d_i$ $(1\le i\le 5)$ satisfy the condition that the points 
$(c_i:d_i)\in\mathbb P^1(\mathbb Q) $ are distinct. Define the linear forms 
$\mathrm M_i=\mathrm M_i (\alpha,\beta)$ $(1\le i\le 5)$ by
\begin{equation}\label{5.1}
\mathrm M_i(\alpha,\beta) = c_i\alpha+d_i\beta.  
\end{equation}
Let $u>0$, and recall the definition of the exponent $\delta(u)$ from \eqref{4.1}. Then, 
with $2P^{4/15}\le Y\le P/16$ and ${\sf n}={\sf n}_{P^4,Y}$, we consider the mean values
\begin{align*} 
I_u=& \int_0^1\!\!\int_0^1 
|f(\mathrm M_1)f(\mathrm M_2)f(\mathrm M_3)|^u\d\alpha\d\beta, \\
J_u=& \int_0^1\!\!\int_0^1 
{\sf n}(\mathrm M_1){\sf n}(\mathrm M_2){\sf n}(\mathrm M_3)
|f(\mathrm M_1)f(\mathrm M_2)f(\mathrm M_3)|^u\d\alpha\d\beta.
\end{align*}

\begin{theorem}\label{thm5.1}
One has $I_4\ll P^{13/2+\eps}$ and $I_u\ll P^{3u-8+\delta(u)+\varepsilon}$ 
$(6\le u\le 25/3)$. Also, when $6\le u\le 11$, one has 
$J_u\ll P^{3u-8+\delta(u)+\varepsilon}$.
\end{theorem}

\begin{proof}
 It follows from Lemmata \ref{minapprox} and \ref{lem3.2} that the 
function ${\sf n}(\gamma)|f(\gamma)|^u$ has a uniformly convergent Fourier series with 
coefficients $\phi_u(n)$. By orthogonality, we conclude that
$$ J_u = \sum_{(n_1,n_2,n_3)\in N} \phi_u(n_1)\phi_u(n_2)\phi_u(n_3), $$
where $N$ is the set of solutions in integers $n_1,n_2,n_3$ of the linear system
$$ c_1n_1+ c_2n_2+ c_3n_3 = d_1n_1+d_2n_2+d_3n_3=0. $$
Since the projective points $(c_i:d_i)$ are distinct, there exist non-zero integers $l_i$, 
depending only on the $c_i,d_i$, having the property that the solutions of this system are 
precisely the triples $(n_1,n_2,n_3)=m(l_1,l_2,l_3)$ $(m\in \mathbb Z)$. It therefore follows 
from \eqref{T} that
$$J_u \le \frac13 \sum_{m\in\mathbb Z} \big(|\phi_u(l_1m)|^3 + |\phi_u(l_2m)|^3 
+|\phi_u(l_3m)|^3 \big) \le  \sum_{n\in\mathbb Z}|\phi_u(n)|^3 .$$
The desired bound for $J_u$ now follows from Theorem \ref{thm4.1}. The bounds for $I_4$ 
and $I_u$ follow in the same way, but the argument has to be built on the cubic moment 
estimates for $\psi_u(n)$ that are provided by Theorems \ref{thm4.1} and \ref{thm4.4}.
\end{proof}

We now turn to related, less balanced mixed moments. With $u$ and $Y$ as before, we 
define
\begin{align*} 
K_u=&\int_0^1\!\!\int_0^1|f(\mathrm M_1)f(\mathrm M_2)|^u |f(\mathrm M_3)|^6\d\alpha
\d\beta, \\
L_u=& \int_0^1\!\!\int_0^1 {\sf n}(\mathrm M_1){\sf n}(\mathrm M_2)
|f(\mathrm M_1)f(\mathrm M_2)|^u |f(\mathrm M_3)|^6\d\alpha\d\beta ,
\end{align*}
and put
$$ \eta(u) = \frac{19}{6}-\frac{u}{3}. $$

\begin{theorem}\label{thm5.2}
Subject to the hypotheses of this section, one has
\begin{align*}
K_u &\ll P^{2u-2+\eta(u)+\varepsilon} \quad (6\le u\le 19/2),\\
L_u &\ll P^{2u-2+\eta(u)+\varepsilon} \quad (6\le u\le 11).
\end{align*}
\end{theorem}
\begin{proof} We proceed as in the initial phase of the proof of Theorem \ref{thm5.1}. 
Using the same notation, we obtain
$$ L_u = \sum_{(n_1,n_2,n_3)\in N} \phi_u(n_1)\phi_u(n_2)\psi_6(n_3). $$
Note here that $\psi_6(m)$ counts solutions of a Diophantine equation, and consequently is 
a non-negative integer. Hence
$$ L_u \le \frac12 \sum_{(n_1,n_2,n_3)\in N} \psi_6(n_3)\big(|\phi_u(n_2)|^2+ 
|\phi_u(n_1)|^2\big). $$
By symmetry, we may therefore suppose that for appropriate non-zero integers $l_2$ and 
$l_3$, depending at most on $\mathbf c$ and $\mathbf d$, one has
\begin{equation}\label{LL}
L_u \le \sum_{(n_1,n_2,n_3)\in N} \psi_6(n_3)|\phi_u(n_2)|^2
= \sum_{m\in\mathbb Z} \psi_6(l_3m)|\phi_u(l_2m)|^2. 
\end{equation}
Next, first applying H\"older's inequality, and then Theorem \ref{thm4.1} and \eqref{4.13}, 
we obtain the bound
\begin{align*}
L_u&\le \Big(\sum_{n\in\mathbb Z} \psi_6(n)^3\Big)^{1/3}
\Big(\sum_{m\in\mathbb Z} |\phi_u(m)|^3\Big)^{2/3}\\
&\ll P^\varepsilon \bigl( P^{15-\frac{23}{6}}\bigr)^{1/3}
\left( P^{\frac{5}{2}u-\frac{23}{6}}\right)^{2/3}.
\end{align*}
The estimate for $L_u$ recorded in Theorem \ref{thm5.2} therefore follows on recalling the 
definition of $\eta(u)$.\par

The initial steps in the estimation of $K_u$ are the same, and one reaches a bound for 
$K_u$ identical to \eqref{LL} except that $\phi_u$ now becomes $\psi_u$. We split into 
major and minor arcs by inserting the relation $1={\sf N}(\alpha)+{\sf n}(\alpha)$, with 
parameters $X=P^4$ and $Y=P^{1/3}$, into \eqref{3.8}. From \eqref{4.51} we obtain
$$ \biggl|\int_0^1 {\sf N}(\alpha) |f(\alpha)|^u e(-\alpha n)\d\alpha\biggr|
\le \int_{{\mathfrak M}_{P^4,P}} |f(\alpha)|^u\d\alpha\ll P^{u-4}.$$
Hence, we discern from \eqref{3.8} and \eqref{3.9} that
$$ |\psi_u(n)|^2\ll |\phi_u(n)|^2 + P^{2u-8}, $$
and so,
$$K_u\ll  \sum_{m\in\mathbb Z} \psi_6(l_3m)|\phi_u(l_2m)|^2
+ P^{2u-8} \sum_{m\in\mathbb Z} \psi_6(l_3m).$$
Here the first sum over $m$ is the same as that occurring in the estimation of $L_u$ in 
\eqref{LL}, and has already been estimated above. Thus, since
$$\sum_{n\in\mathbb Z}\psi_6(n)= |f(0)|^6 \ll P^6,$$
we conclude that
$$K_u\ll P^{2u-2+\eta(u)+\varepsilon}+P^{2u-8}\sum_{n\in\mathbb Z} \psi_6(n)\ll 
P^{2u-2+\eta(u)+\varepsilon}+P^{2u-2}.$$
Provided that $u\le 19/2$, which guarantees $\eta(u)$ to be non-negative, this estimate 
confirms the upper bound for $K_u$ claimed in the theorem.
\end{proof}

Note that the mean values $I_u$ and $J_u$ involve $s= 3u$ Weyl sums, at least for integral 
values of $u$. By comparison, the number of Weyl sums in $K_u$ and $L_u$ is $s=2u+6$. 
A short calculation shows that when applied with the same value of $s$, with $s\ge 18$, the 
exponents of $P$ in Theorems \ref{thm5.1} and \ref{thm5.2} coincide. Since almost all of 
Theorem \ref{thm5.1} may be recovered from Theorem \ref{thm5.2} via H\"older's 
inequality, and since for fixed values of $s$ the exponent $u$ in Theorem \ref{thm5.2} is at 
least as large, Theorem \ref{thm5.2} is morally the stronger result. In our later application 
of the circle method, this allows for larger values of $r_j$ in the profiles associated to the 
simultaneous equations \eqref{1.1}, and this is essential for our method to succeed. 
Another advantage is that in $L_u$ only two of the forms $\mathrm M_i$ are on minor arcs, 
while in the mean value $J_u$ all three are constrained to minor arcs.\par

We continue with another result in which the profile is even farther out of balance. We 
consider the integral 
$$ M = \int_0^1\!\!\int_0^1 {\sf n}(\mathrm M_1) {\sf n}(\mathrm M_2)
|f(\mathrm M_1)^{11}f(\mathrm M_2)^{11}f(\mathrm M_3)^4|\d\alpha\d\beta. $$

\begin{theorem}\label{thm5.3}
Given the hypotheses of this section, one has $M \ll P^{18-1/{18}+\varepsilon}$. 
\end{theorem}

\begin{proof}
We again traverse the initial phase of the proof of Theorem \ref{thm5.1} to confirm the 
relation
$$ M = \sum_{(n_1,n_2,n_3)\in N} \phi_{11}(n_1)\phi_{11}(n_2)\psi_4(n_3). $$
Then, just as in the argument of the proof of Theorem \ref{thm5.2} leading to \eqref{LL}, 
we find that for appropriate non-zero integers $l_2$ and $l_3$, depending at most on 
$\mathbf c$ and $\mathbf d$, one has
$$ M \le \sum_{m\in\mathbb Z} \psi_4(l_3m)|\phi_{11}(l_2m)|^2.$$
Thus, an application of H\"older's inequality in combination with Theorems \ref{thm4.1} and 
\ref{thm4.4}, together with \eqref{4.13}, yields the bound
$$M\le \Bigl( \sum_{n\in \mathbb Z} \psi_{4}(n)^{3} \Bigr)^{1/3}
\Bigl( \sum_{n\in \mathbb Z} |\phi_{11}(n)|^{3} \Bigr)^{2/3}\ll P^\varepsilon 
\bigl( P^{13/2}\bigr)^{1/3}\left( P^{71/3}\right)^{2/3}.$$
The desired conclusion follows a rapid computation.\end{proof}

Finally, we transform the estimates for $L_u$ and $M$ into proper minor arc estimates. In 
the interest of brevity we write $\mathfrak M= \mathfrak M_{P^4,P^{1/3}}$ and put
\begin{equation}\label{pee}
\mathfrak p = [0,1]^2 \setminus (\mathfrak M\times \mathfrak M).
\end{equation}

\begin{theorem}\label{thm5.4}
Suppose that $19/2 < u \le 11$. Then
\begin{equation}\label{minor6uu}
\iint_{\mathfrak p}|f(\mathrm M_1)f(\mathrm M_2)|^u |f(\mathrm M_3)|^6\d\alpha\d\beta 
\ll P^{2u-2+\eta(u)+\varepsilon}.
\end{equation}
Further, one has
\begin{equation}\label{minor41111}
\iint_{\mathfrak p}|f(\mathrm M_1)^{11}f(\mathrm M_2)^{11}f(\mathrm M_3)^4|
\d\alpha\d\beta \ll P^{18-1/18+\varepsilon}.
\end{equation}
\end{theorem}

\begin{proof}
Let ${\sf N}= {\sf N}_{P^4,P^{2/7}} $ and ${\sf n}=1-{\sf N}$. Then
\begin{equation}\label{8.4}
1 = \big({\sf N}(\mathrm M_1) + {\sf n}(\mathrm M_1)\big)\big({\sf N}(\mathrm M_2) 
+ {\sf n}(\mathrm M_2)\big).
\end{equation}
We note at once that whenever $(\alpha,\beta)\in\mathfrak p$, one has 
${\sf N}(\mathrm M_1){\sf N}(\mathrm M_2)=0$. The explanation for this observation is 
that whenever ${\sf N}(\mathrm M_1){\sf N}(\mathrm M_2)>0$, then it follows from 
Lemma \ref{majapprox} that $\mathrm M_j\in\mathfrak N_{P^4,2P^{2/7}}$ $(j=1,2)$. By 
taking suitable linear combinations of $\mathrm M_1$ and $\mathrm M_2$ we find that 
$\alpha$ and $\beta$ lie in $\mathfrak N_{P^4,AP^{2/7}}$, with some $A\ge 2$ depending 
only on the coefficients of $\mathrm M_1$ and $\mathrm M_2$. But 
$(\alpha,\beta)\in[0,1]^2$, and so  $(\alpha,\beta)\in\mathfrak M\times\mathfrak M$ for 
large enough $P$. This is not the case when $(\alpha,\beta)\in\mathfrak p$, as claimed.\par

With this observation in hand, we apply \eqref{8.4} within the integral on the left hand side 
of \eqref{minor41111} to conclude that
\begin{equation}\label{8.5}
\iint_{\mathfrak p}|f(\mathrm M_1)^{11}f(\mathrm M_2)^{11}f(\mathrm M_3)^4|
\d\alpha\d\beta \le M  + M_{\sf Nn} + M_{\sf nN},
\end{equation}
where
\begin{equation}\label{8.7}
M_{\sf Nn} = \int_0^1\!\!\int_0^1 {\sf N}(\mathrm M_1){\sf n}(\mathrm M_2)
|f(\mathrm M_1)^{11}f(\mathrm M_2)^{11} f(\mathrm M_3)^4| \,\mathrm d\alpha\,
\mathrm d\beta
\end{equation}
and $M_{\sf nN}$ is the integral in \eqref{8.7} with $\mathrm M_1$, $\mathrm M_2$ 
interchanged.\par

By symmetry in $\mathrm M_1$ and $\mathrm M_2$, it now suffices to estimate 
$M_{\sf Nn}$. Recalling the definition \eqref{5.1} of the linear forms $\mathrm M_i$, we put 
$D=|c_1d_2-c_2d_1|$ and note that $D>0$. Consider the linear transformation from 
$\mathbb R^2$ to $\mathbb R^2$, with $(\alpha,\beta)\mapsto (\alpha',\beta')$, defined by 
means of the relation
\begin{equation}\label{8.8}
\Big(\begin{array}{c} \alpha' \\ \beta'\end{array}\Big)=D^{-1}
\Big(\begin{array}{cc} c_1&d_1 \\ c_2& d_2\end{array}\Big)\Big(\begin{array}{c} \alpha \\ 
\beta\end{array}\Big). 
\end{equation}
Then $\mathrm M_1= D\alpha'$, $\mathrm M_2=D\beta'$, and $\alpha$ and $\beta$ are 
linear forms in $\alpha'$ and $\beta'$ with integer coefficients. By applying the 
transformation formula as a change of variables, one finds that
$$ M_{\sf Nn} = \iint_{\mathfrak B} {\sf N}(D\alpha'){\sf n}(D\beta')
|f(D\alpha')^{11}f(D\beta')^{11} f(A\alpha'+B\beta')^4|  \,\mathrm d\alpha'\,
\mathrm d\beta', $$
wherein $A,B$ are non-zero integers and $\mathfrak B$ is the image of $[0,1]^2$ under the 
transformation \eqref{8.8}. The parallelogram $\mathfrak B$ is covered by finitely many 
sets $[0,1]^2+\mathbf t$, with $\mathbf t\in\mathbb Z^2$. Since the integrand in the last 
expression for $M_{\sf Nn}$ is $\mathbb Z^2$-periodic it follows that
$$ M_{\sf Nn} \ll  \int_0^1\!\!\int_0^1 {\sf N}(D\alpha){\sf n}(D\beta)
|f(D\alpha)^{11}f(D\beta)^{11}f(A\alpha+B\beta)^4|\,\mathrm d\alpha\,\mathrm d\beta. 
$$
Here we have removed decorations from the variables of integration for notational simplicity. 

\par We now inspect all factors of the integrand in the latter upper bound that depend on 
$\beta$. By H\"older's inequality, Lemma \ref{lemma5.3} and obvious changes of variable, 
one obtains the estimate
\begin{align*}
\int_0^1  {\sf n}(D\beta)&|f(D\beta)^{11}f(A\alpha+B\beta)^4|\,\mathrm d\beta \\
&\ll  \biggl( \int_0^1  {\sf n}(D\beta)|f(D\beta)|^{77/5}\,\mathrm d\beta\biggr)^{5/7} 
\biggl( \int_0^1 |f(A\alpha+B\beta)|^{14}\,\mathrm d\beta \biggr)^{2/7}\\
& \ll P^\varepsilon \bigl( P^{67/6}\bigr)^{5/7}(P^{10})^{2/7}=P^{65/6+\varepsilon},
\end{align*}
uniformly in $\alpha\in\mathbb R$. Consequently, applying \eqref{4.51} in combination with 
yet another change of variable, we finally arrive at the bound 
$$ M_{\sf Nn} \ll P^{65/6+\varepsilon} \int_0^1 {\sf N}(D\alpha)|f(D\alpha)|^{11}
\,\mathrm d\alpha \ll P^{18-1/6+\varepsilon}. $$
We may infer thus far that $M_{\sf Nn}+M_{\sf nN}\ll P^{18-1/6+\varepsilon}$. On 
substituting this estimate into \eqref{8.5}, noting also the bound 
$M\ll P^{18-1/18+\varepsilon}$ supplied by Theorem \ref{thm5.3}, the conclusion 
\eqref{minor41111} is confirmed.\par

The proof of \eqref{minor6uu} is essentially the same, and we economise by making similar 
notational conventions. The exponents $11$ and $4$ that occur in \eqref{minor41111} must 
now be replaced by $u$ and $6$, respectively. The initial phase of the preceding argument 
then remains valid, and an appeal to Theorem \ref{thm5.2} delivers the bound
\begin{equation}\label{8.z}
\iint_{\mathfrak p}|f(\mathrm M_1)f(\mathrm M_2)|^u|f(\mathrm M_3)|^6
\d\alpha\d\beta \ll L_{\sf Nn} + L_{\sf nN}+P^{2u-2+\eta(u)+\varepsilon},
\end{equation}
where 
$$ L_{\sf Nn} \ll \int_0^1\!\!\int_0^1 {\sf N}(D\alpha){\sf n}(D\beta)   
|f(D\alpha)f(D\beta)|^{u} |f(A\alpha+B\beta)|^6  \,\mathrm d\alpha\,\mathrm d\beta. $$
Here, we isolate factors of the integrand that depend on $\beta$ and apply H\"older's 
inequality. Note that since $u\le 11$ we have $7u/4<20$. Thus, by Lemma \ref{lemma5.3}, 
\begin{align*}
\int_0^1  {\sf n}(D\beta)&|f(D\beta)^{u}f(A\alpha+B\beta)^6|\,\mathrm d\beta \\ 
&\ll \biggl( \int_0^1  {\sf n}(D\beta)|f(D\beta)|^{7u/4}\,\mathrm d\beta\biggr)^{4/7}
\biggl( \int_0^1 |f(A\alpha+B\beta)|^{14}\,\mathrm d\beta \biggr)^{3/7}\\
&\ll P^\varepsilon \bigl( P^{\frac{35}{24}u-\frac{5}{3}}\bigr)^{4/7}
\bigl( P^{10}\bigr)^{3/7}.
\end{align*}
Applying this bound, which is uniform in $\alpha \in \mathbb R$, together with \eqref{4.51}, 
we arrive at the estimate
$$L_{\sf Nn}\ll P^{\frac56 u+\frac{10}{3}+\varepsilon}
\int_0^1 {\sf N}(D\alpha)|f(D\alpha)|^u\,\mathrm d\alpha
\ll P^{\frac{11}6 u - \frac{2}3+\varepsilon}.$$
When $u\le 11$, the definition of $\eta(u)$ ensures that 
$\frac{11}{6}u-\frac{2}{3}\le 2u-2+\eta(u)$, and hence 
$L_{\sf Nn}+L_{\sf nN}\ll P^{2u-2+\eta(u)+\varepsilon}$. The conclusion 
\eqref{minor6uu} now follows by substituting this estimate into \eqref{8.z}.
\end{proof}

\section{Another mean value estimate}
This section is an update for quartic Weyl sums of our earlier work \cite{BWBull} on highly 
entangled mean values. We now attempt to avoid independence conditions on linear forms 
as far as the argument allows while incorporating the consequences of the recent bound 
\eqref{1.5}. We emphasise that throughout this section, we continue to work subject to the 
overall assumptions made at the outset of the previous section. We begin by examining the 
mean value
\begin{equation}\label{7.z1}
G_1=\int_0^1\!\!\int_0^1|f(\mathrm M_1)^2 f(\mathrm M_2)^4
f(\mathrm M_3)^4|\d\alpha\d\beta .
\end{equation}

\begin{lemma}\label{lemma6.1}
One has $G_1\ll P^{5+\varepsilon}$.
\end{lemma} 

\begin{proof}
This is essentially contained in \cite[Section 2]{BWpauc}, but we give a proof for 
completeness. Recall the definition \eqref{5.1} of the linear forms $\mathrm M_i$. By 
orthogonality, the integral $G_1$ is equal to the number of solutions of an associated 
pair of quartic equations. By taking suitable integral linear combinations of these two 
equations, we may assume that they take the shape
\begin{equation}\label{10var}
a(x_1^4-x_2^4)=b(x_3^4+x_4^4-x_5^4-x_6^4)=c(x_7^4+x_8^4-x_9^4-x_{10}^4), 
\end{equation}
for suitable natural numbers $a,b,c$. Thus, we see that $G_1$ is equal to the number of 
solutions of the Diophantine system \eqref{10var} with $x_i\le P$. For each of the $O(P)$ 
possible choices for $x_1$ and $x_2$ with $x_1=x_2$, it follows via orthogonality and 
\eqref{Hua4} that the number of solutions of this system in the remaining variables 
$x_3,\ldots,x_{10}$ is equal to
$$\biggl( \int_0^1 |f(\alpha)|^4\,\d\alpha\biggr)^2 \ll P^{4+\varepsilon}.$$
Consequently, the contribution to $G_1$ from this first class of solutions is 
$O(P^{5+\varepsilon})$. Now consider solutions of \eqref{10var} in which $x_1\neq x_2$. 
By orthogonality, the total number of choices for $x_3, \ldots, x_{10}$ satisfying the 
rightmost equation in \eqref{10var} is
$$ \int_0^1 |f(b\alpha)f(c\alpha)|^4\d\alpha. $$
Schwarz's inequality in combination with \eqref{4.6} shows this integral to be 
$O(P^{5+\varepsilon})$. However, for any fixed choice of $x_3,\ldots,x_{10}$ in this 
second class of solutions, one has $x_1\ne x_2$, and hence the fixed integer 
$N=b(x_3^4+x_4^4-x_5^4-x_6^4)$ is non-zero. But it follows from \eqref{10var} that 
$x_1^2-x_2^2$ and $x_1^2+x_2^2$ are each divisors of $N$. Thus, a standard divisor 
function estimate shows that the number of choices for $x_1$ and $x_2$ is 
$O(P^\varepsilon)$, and we conclude that the contribution to $G_1$ from this second class 
of solutions is $O(P^{5+\eps})$. Adding these two contributions, we obtain the bound 
claimed in the statement of the lemma.
\end{proof}

We next examine the mean value
\begin{equation}\label{7.z3}
G_2=\int_0^1\!\!\int_0^1 |f(\mathrm M_1)^2 f(\mathrm M_2)^4f(\mathrm M_3)^4
f(\mathrm M_4)^4f(\mathrm M_5)^4|\d\alpha\d\beta .
\end{equation}

\begin{theorem}\label{thm6.2} 
One has $G_2 \ll P^{11+\varepsilon}$. 
\end{theorem}

Note that in this result we require the five linear forms $\mathrm M_j$ to be pairwise 
independent. Therefore, the result will be of use only in cases where the profile of 
\eqref{1.1} has $r_5\ge 1$. The mean value in Theorem \ref{thm6.2} involves $18$ Weyl 
sums and should therefore be compared with the bound $I_6 \ll P^{67/6+\varepsilon}$ 
provided by Theorem \ref{thm5.1}. The extra savings that we obtain here are the essential 
stepping stone toward Theorem \ref{theorem1.2}.

\begin{proof}[The proof of Theorem \ref{thm6.2}]
As in the proof of Lemma \ref{lemma6.1}, it follows from orthogonality that the integral 
$G_2$ is equal to the number of solutions of an associated pair of quartic equations. Taking 
suitable integral linear combinations of these two equations, we reduce to the situation 
where $c_4=d_5=0$, and consequently $\mathrm M_4 = d_4\beta$ and 
$\mathrm M_5=c_5\alpha$. Motivated by this observation, we begin our deliberations by 
estimating the auxiliary mean value
$$G_3=\int_0^1\!\!\int_0^1 |f(\mathrm M_1)^2 f(\mathrm M_2)^4
f(\mathrm M_3)^4f(d_4\beta)^4|\d\alpha\d\beta .$$

\par The Weyl differencing argument \cite[Lemma 2.3]{hlm} shows that there are real 
numbers $u_h$ with $u_h\ll P^\varepsilon$ for which
\begin{equation}\label{6.4}
|f(\gamma)|^4 \ll P^3 + P \sum_{1\le |h|\le 2P^4} u_h e(\gamma h).
\end{equation} 
We apply this relation with $\gamma=\mathrm M_4$ to the mean value $G_3$ and infer 
that
\begin{equation}\label{7.z2}
G_3\ll P^3 G_1 + PG_4,
\end{equation}
where $G_1$ is the mean value defined in \eqref{7.z1}, and
$$ G_4 = \sum_{1\le |h|\le 2P^4} u_h\int_0^1\!\!\int_0^1 |f(\mathrm M_1)^2 
f(\mathrm M_2)^4f(\mathrm M_3)^4|e(d_4 h\beta) \d\alpha\d\beta. $$
By orthogonality, the double integral on the right hand side here is equal to the number of 
solutions of the system of Diophantine equations
\begin{align}
c_1(x_1^4-y_1^4) + c_2(x_2^4+x_3^4-y_2^4-y_3^4)+ c_3(x_4^4+x_5^4-y_4^4-y_5^4) 
&\hskip-2.5mm&= 0 \label{6.1}\\
d_1(x_1^4-y_1^4) + d_2(x_2^4+x_3^4-y_2^4-y_3^4)+ d_3(x_4^4+x_5^4-y_4^4-y_5^4) 
&\hskip-2.5mm&+\hskip.3mm d_4h=0 \notag
\end{align}
with $x_i\le P$ and $y_i\le P$. We may sum over $h\neq 0$ and replace $u_h$ by its upper 
bound. Then we find that $G_4\ll P^\varepsilon G_5$, where $G_5$ is the number of 
solutions of the equation \eqref{6.1} with the same conditions on $x_i$ and $y_i$. By 
orthogonality again, we deduce that
$$G_5=\int_0^1 |f(c_1\alpha)^2f(c_2\alpha)^4f(c_3\alpha)^4|\d\alpha. $$
For $1\le i\le 3$ the linear form $\mathrm M_i$ is linearly independent of 
$\mathrm M_4=d_4\beta$, and thus $c_1c_2c_3\neq 0$. The trivial bound 
$|f(c_1\alpha)|^2\ll P^2$ therefore combines with Schwarz's inequality and \eqref{4.6} to 
award us the bound
$$G_5\ll P^2\int_0^1 |f(\gamma)|^8\d\gamma \ll P^{7+\varepsilon}. $$
We therefore deduce that $G_4\ll P^{7+2\varepsilon}$. Meanwhile, the estimate 
$G_1\ll P^{5+\varepsilon}$ is available from Lemma \ref{lemma6.1}. On substituting these 
bounds into \eqref{7.z2}, we conclude thus far that $G_3\ll P^{8+\varepsilon}$.\par

We now repeat this argument with $\gamma=\mathrm M_5$ in \eqref{6.4}, applying the 
resulting inequality within the integral $G_2$ defined in \eqref{7.z3}. Thus we obtain
\begin{equation}\label{7.z4}
G_2\ll P^3G_3+P^{1+\varepsilon}G_6,
\end{equation}
where $G_6$ denotes the number of solutions of the Diophantine equation
$$d_1(x_1^4-y_1^4)+d_2(x_2^4+x_3^4-y_2^4-y_3^4)+d_3(x_4^4+x_5^4-y_4^4-y_5^4) 
+ d_4(x_6^4+x_7^4-y_6^4-y_7^4)=0,$$
with $x_i\le P$ and $y_i\le P$. By orthogonality,
$$G_6=\int_0^1 |f(d_1\alpha)^2f(d_2\alpha)^4f(d_3\alpha)^4f(d_4\alpha)^4|\d\alpha.$$
One may confirm that $d_1d_2d_3d_4\neq 0$ by arguing as above, and so an application of 
\eqref{T} in combination with \eqref{1.5} reveals that
$$G_6\le \sum_{i=1}^4 \int_0^1 |f(d_i\alpha)|^{14} \d\alpha = 
4 \int_0^1 |f(\gamma)|^{14}\d\gamma\ll P^{10+\varepsilon}. $$
The conclusion of 
the theorem now follows on substituting this bound together with our earlier estimate for 
$G_3$ into \eqref{7.z4}.
\end{proof}

\section{The circle method} In this section we prepare the ground to advance to the proofs 
of Theorems \ref{theorem1.1} and \ref{theorem1.2}. A preliminary man\oe uvre is in order. 
Let $k=0$ or 1, and let $N_k(P)=N_k$ denote the number of solutions of the system 
\eqref{1.1} with $k\le x_j\le P$ $(1\le j\le s)$. Note that the equations \eqref{1.1} are 
invariant under the $s$ mappings $x_j\mapsto -x_j$. This observation shows that
\begin{equation}
\label{sandw}  2^s N_1(P) \le {\mathscr N}(P) \le 2^s N_0 (P).
\end{equation}
The goal is then to establish the formulae
\begin{equation}\label{sandw2}
\lim_{P\to\infty} 2^s P^{8-s} N_k(P) =  \mathfrak I \mathfrak S  \quad (k=0, 1),
\end{equation}
since then \eqref{1.4} follows immediately from \eqref{sandw} and the sandwich principle. 
Thus, we now launch the Hardy-Littlewood method to evaluate the counting functions 
$N_k(P)$. This involves the exponential sum
\begin{equation} \label{ffk}
f_k(\alpha)=\sum_{k\le x\le P} e(\alpha x^4).
\end{equation}
This sum is, of course, an instance of the sum \eqref{ff}, where we have been deliberately 
imprecise about the lower end of the interval of summation. The results we have formulated 
so far are indeed independent of the choice of $k$, and it is only now and temporarily 
where this detail matters. We require the linear forms 
$\Lambda_j=\Lambda_j(\alpha,\beta)$, defined by
$$\Lambda_j(\alpha,\beta)=a_j\alpha + b_j \beta\quad (1\le j\le s)$$
that are associated with the equations \eqref{1.1}. We then put
\begin{equation}\label{7.2}
{\mathscr F}_k(\alpha,\beta) = f_k(\Lambda_1)f_k(\Lambda_2)\cdots f_k(\Lambda_s),
\end{equation}
and observe that, by orthogonality, one has
\begin{equation}\label{7.3}
N_k(P)=\int_0^1\!\!\int_0^1{\mathscr F}_k(\alpha,\beta)\,\mathrm d\alpha
\,\mathrm d\beta.
\end{equation} 

Subject to conditions milder than those imposed in Theorems \ref{theorem1.1} and 
\ref{theorem1.2} we reduce the evaluation of the integral \eqref{7.3} to the estimation of 
its minor arc part. With this end in mind we define the major arcs $\mathfrak V$ as the 
union of the rectangles
$$ \mathfrak V(q,a,b)=\{(\alpha,\beta)\in[0,1]^2: \text{$|\alpha-a/q|\le P^{-31/8}$ and 
$|\beta-b/q|\le P^{-31/8}$}\},$$
with $0\le a,b\le q$, $(a,b,q)=1$ and $1\le q\le P^{1/8}$.\par   

Define the generating functions
$$ S(q,c) = \sum_{x=1}^q e(cx^4/q)\quad \text{and}\quad 
v(\gamma) = \int_0^P e(\gamma t^4)\,\mathrm d t. $$
Then, given $(\alpha,\beta)\in [0,1]^2$, if we put $\gamma=\alpha -a/q$ and 
$\delta=\beta-b/q$ for some $a,b\in\mathbb Z$ and $q\in\mathbb N$, one concludes from 
\eqref{ffk} and \cite[Theorem 4.1]{hlm} that
\begin{equation}\label{Fapprox}
f_k(\Lambda_j)=q^{-1}S\left(q,\Lambda_j(a,b)\right)v\left(\Lambda_j(\gamma,\delta)\right)
+ O\left(q^{1/2+\varepsilon}(1+P^4|\Lambda_j(\gamma,\delta)|)^{1/2}\right). 
\end{equation}
Note that the right hand side here is independent of $k$. We multiply these approximations 
for $1\le j\le s$. This brings into play the expressions
$$\mathscr S(q,a,b)=q^{-s} \prod_{j=1}^s S\left(q,\Lambda_j(a,b)\right) \quad 
\text{and}\quad \mathscr V(\gamma,\delta) = \prod_{j=1}^s 
v\left(\Lambda_j(\gamma,\delta)\right).$$ 
If $(\alpha,\beta)\in \mathfrak V(q,a,b)\subseteq \mathfrak V$ then the error term in 
\eqref{Fapprox} is $O(P^{1/8+\varepsilon})$, and we infer that
$$\mathscr F_k(\alpha,\beta)=\mathscr S(q,a,b)\mathscr V(\gamma,\delta)+
O(P^{s-7/8+\varepsilon}). $$
Since $\mathfrak V$ is a set of measure $O(P^{-59/8})$, when we integrate this formula 
for $\mathscr F_k(\alpha, \beta)$ over $\mathfrak V$, we obtain  the asymptotic relation
$$\iint_{\mathfrak V} \mathscr F_k(\alpha,\beta) \,\mathrm d\alpha\, \mathrm d\beta =
\mathfrak S(P^{1/8})\mathfrak J^*(P^{1/8})+O(P^{s-33/4+\varepsilon}),$$
where, for $1\le Q\le P$ we define
\begin{align*}
\mathfrak S (Q)&=\sum_{q\le Q}\underset{(a,b,q)=1}{\sum_{a=1}^q\sum_{b=1}^q}
\mathscr S(q,a,b),\\
\mathfrak J^*(Q)&=\iint_{\mathfrak U(Q)} \mathscr V(\gamma,\delta)\,\mathrm d\gamma\, 
\mathrm d\delta ,
\end{align*}
and $\mathfrak U(Q)=[-QP^{-4},QP^{-4}]^2$.\par

At this point, we require some more information concerning the matrix of coefficients, and 
we shall suppose that $q_0\ge 15$. Then $s\ge 16$, and we may apply 
\cite[Lemma 3.3]{BW21} to conclude that 
$\mathfrak S(Q)=\mathfrak S+O(Q^{\varepsilon -1})$. Further, we have
$$ \int_{-P}^P e(\gamma t^4) \,\mathrm dt = 2 v(\gamma),$$
and thus \cite[Lemma 3.1]{BW21} shows that the limit \eqref{1.2} exists, and that we 
have $2^s\mathfrak J^*(Q)=P^{s-8}\mathfrak J+O(P^{s-8}Q^{-1/4})$.
We summarise these deliberations in the following lemma.

\begin{lemma}\label{lem7.1} Suppose that $q_0\ge 15$ and that $k\in \{0,1\}$. Then
$$ \iint_{\mathfrak V} \mathscr F_k(\alpha,\beta) \,\mathrm d\alpha\, \mathrm d\beta = 
2^{-s}P^{s-8}\mathfrak S\mathfrak J + O(P^{s-8-1/32}). $$
\end{lemma}
 
The major arcs in Lemma \ref{lem7.1} are certainly too slim for efficient use of Weyl type 
inequalities on the complementary set. A pruning argument allows us to enlarge the major 
arcs considerably. Let $\mathfrak W$ denote the union of the rectangles 
$$\mathfrak W(q,a,b)=\{(\alpha,\beta)\in[0,1]^2: \text{$|q\alpha-a|\le P^{-3}$ and 
$|q\beta-b|\le P^{-3}$}\},$$
with $1\le q\le P$, $0\le a,b \le q$ and $(a,b,q)=1$. Then 
$\mathfrak V\subset \mathfrak W$, and we proceed to estimate the contribution from 
$\mathfrak W\setminus \mathfrak V$ to the integral \eqref{7.3}. 
A careful application of \cite[Theorem 4.2]{hlm} shows that $S(q,c) \ll q^{3/4}(q,c)^{1/4}$. 
Further, if $V(\gamma) = P(1+P^4|\gamma|)^{-1/4}$, then by \cite[Theorem 7.3]{hlm}, 
one has $v(\gamma)\ll V(\gamma)$. Hence, whenever 
$(\alpha,\beta)\in \mathfrak W(q,a,b)$ with $q\le P$, one deduces from \eqref{Fapprox} 
that 
$$f_k(\Lambda_j)\ll q^{-1/4}\left( q,\Lambda_j(a,b)\right)^{1/4} 
V\left( \Lambda_j(\alpha-a/q,\beta-b/q)\right) + P^{1/2+\varepsilon}.$$
It is immediate that the first term on the right hand side here always dominates the second, 
and therefore, 
$${\mathscr F}_k(\alpha,\beta)\ll q^{-s/4}\prod_{j=1}^s
\left( q,\Lambda_j(a,b)\right)^{1/4} V\left( \Lambda_j(\alpha-a/q,\beta-b/q)\right). $$

\par We integrate over $\mathfrak W\setminus \mathfrak V$. The result is a sum over 
$q\le P$ in which we consider the portion $q\le P^{1/8}$ separately. This yields the bound
\begin{equation}\label{8.z1}
\iint_{\mathfrak W\setminus\mathfrak V}\mathscr F_k(\alpha,\beta)\,\mathrm d\alpha\, 
\mathrm d\beta \ll  K_1(P^{1/8})+K_2(P^{1/8}),
\end{equation}
where for $1\le Q\le P$, we write
$$K_1(Q)=\sum_{q\le Q}\underset{(a,b,q)=1}{\sum_{a=1}^q\sum_{b=1}^q} q^{-s/4}
\prod_{j=1}^s \left( q,\Lambda_j(a,b)\right)^{1/4}\iint_{\mathfrak B(Q)}
\prod_{j=1}^sV(\Lambda_j) \,\mathrm d\alpha\,\mathrm d\beta ,$$
with $\mathfrak B(Q)=[-1,1]^2\setminus \mathfrak U(Q)$, and
$$K_2(Q)=\sum_{Q<q\le P}\underset{(a,b,q)=1}{\sum_{a=1}^q\sum_{b=1}^q}
q^{-s/4}\prod_{j=1}^s (q,\Lambda_j(a,b))^{1/4}\iint_{[-1,1]^2}
\prod_{j=1}^s V(\Lambda_j) \,\mathrm d\alpha\,\mathrm d\beta.$$
Still subject to the condition $q_0\ge 15$, the proof of \cite[Lemma 3.2]{BW21} shows that
$$ \sum_{q>Q}\underset{(a,b,q)=1}{\sum_{a=1}^q\sum_{b=1}^q} q^{-s/4}
\prod_{j=1}^s(q,\Lambda_j(a,b))^{1/4} \ll \sum_{q>Q}q^{\varepsilon-2}\ll 
Q^{\varepsilon-1},$$
and similarly, the proof of \cite[Lemma 3.1]{BW21} delivers the bound
$$ \iint_{\mathfrak B(Q)}\prod_{j=1}^s V(\Lambda_j)\,\mathrm d\alpha\,\mathrm d\beta 
\ll P^{s-8}Q^{-1/4}. $$
Thus we deduce that $K_1(P^{1/8})+K_2(P^{1/8})\ll P^{s-8-1/32}$. Substituting this 
estimate into \eqref{8.z1}, and then recalling Lemma \ref{lem7.1}, we see that in the latter 
lemma we may replace $\mathfrak V$ by $\mathfrak W$. This establishes the following 
theorem.

\begin{theorem}\label{thm7.2}
Suppose that $q_0\ge 15$ and that $k\in \{0,1\}$. Then
$$\iint_{\mathfrak W}\mathscr F_k(\alpha,\beta)\,\mathrm d\alpha\, \mathrm d\beta = 
2^{-s} P^{s-8}\mathfrak S\mathfrak I + O(P^{s-8-1/32}). $$
\end{theorem}

Let $\mathfrak w= [0,1]^2\setminus \mathfrak W$ denote the minor arcs. Then, in view of 
\eqref{sandw2}, \eqref{7.3} and Theorem \ref{thm7.2}, whenever $q_0\ge 15$, the 
asymptotic relation \eqref{1.4} is equivalent to the minor arc estimate
\begin{equation}\label{7minor}
\iint_{\mathfrak w}\mathscr F_k(\alpha,\beta) \,\mathrm d\alpha\, \mathrm d\beta = 
o(P^{s-8}),
\end{equation} 
as $P\to \infty$, and in the next two sections we shall confirm this subject to the hypotheses 
imposed in Theorems \ref{theorem1.1} and \ref{theorem1.2}.

\section{The proof of Theorem \ref{theorem1.1}} 
At the core of the proof of Theorem \ref{theorem1.1} we require two minor arc estimates.

\begin{lemma}\label{lem8.1}
Let $c_1,c_2,d_1,d_2\in\mathbb Z$, and suppose that $\mathrm M_j=c_j\alpha+d_j\beta$ 
$(j=1,2)$ are linearly independent. Then
$$ \iint_{\mathfrak w} |f(\mathrm M_1)f(\mathrm M_2)|^{15}\,\mathrm d\alpha\,
\mathrm d\beta \ll P^{22-1/6+\varepsilon}. $$
\end{lemma}

\begin{proof} It is immediate from \eqref{pee} that $\mathfrak w \subset \mathfrak p$. 
Recall the initial  argument within the proof of Theorem \ref{thm5.4}. This shows that for 
$(\alpha,\beta)\in \mathfrak p$, the forms $\mathrm M_1$ and $\mathrm M_2$ cannot be 
in $\mathfrak N_{P^4,P^{2/7}}$ simultaneously. By symmetry we may therefore suppose 
that $\mathrm M_1 \in\mathfrak n_{P^4,P^{2/7}}$. Now apply the transformation formula 
as in \eqref{8.8}. One finds that for an appropriate non-zero integer $D$, depending at 
most on $\mathbf c$ and $\mathbf d$, one has 
$$\iint_{\mathfrak w}|f(\mathrm M_1)f(\mathrm M_2)|^{15}\,\mathrm d\alpha\, 
\mathrm d\beta \ll \int_0^1\!\!\int_{\mathfrak m} |f(D\alpha )f(D\beta)|^{15}\,
\mathrm d\alpha\,\mathrm d\beta ,$$
where $\mathfrak m = \mathfrak m_{P^4,P^{2/7}}$. Thus, applying a trivial estimate for 
one factor $f(D\beta)$, we deduce via Lemma \ref{lemma5.3} that
$$\iint_{\mathfrak w}|f(\mathrm M_1)f(\mathrm M_2)|^{15}\,\mathrm d\alpha\, 
\mathrm d\beta \ll P^\varepsilon \left( P^{65/6}\right) \left( P^{11}\right) 
\ll P^{22-1/6+\varepsilon}.$$
This completes the proof of the lemma.
\end{proof}

\begin{lemma}\label{lem8.2}
Suppose that any two of the binary linear forms $\mathrm M_1$, $\mathrm M_2$, 
$\mathrm M_3$ are linearly independent. Then
$$ \iint_{\mathfrak w} |f(\mathrm M_1)^{11}f(\mathrm M_2)^{11} f(\mathrm M_3)^4 |  \,\mathrm d\alpha\,\mathrm d\beta \ll P^{18-1/18+\varepsilon}. $$
\end{lemma} 

\begin{proof} On recalling that $\mathfrak w \subset \mathfrak p$, the lemma is immediate 
from Theorem \ref{thm5.4}.\end{proof}

We are now fully equipped to complete the proof of Theorem \ref{theorem1.1}. Suppose 
that we are given a pair of equations \eqref{1.1} with $s\ge 26$, $q_0\ge 15$ and profile 
$(r_1,r_2,\ldots,r_\nu)$. The parameter $l=s-r_1-r_2$ determines our argument. In the 
notation of Section 7, we let $\mathscr F= \mathscr F_k$ with $k=0$ or $1$ be the 
generating function defined in \eqref{7.2}.\par

Small values of $l$ call for special attention. Initially, we consider the situation with 
$0\le l \le 3$. We apply Lemma \ref{lemA3} with $J_1$ and $J_2$ the subsets of the set of 
indices $\{1,2,\ldots ,s\}$ counted by $r_1$ and $r_2$, respectively, and with $J_3$ the 
subset consisting of the remaining indices. Then $\text{card}(J_3)=l$. We also choose 
$$M_\nu=\ldots =M_4=0,\quad M_3=l,\quad M_2=15-l\quad \text{and}\quad M_1=s-15.$$
The condition 
$q_0\ge 15$ ensures that 
$r_1\le s-15$, and $r_1+r_2=s-l=M_1+M_2$. Also, we have $M_1=s-15\ge 15-l=M_2$ 
because $r_1\ge r_2 \ge 15-l$ and $s=r_1+r_2+l\ge 2r_2 +l \ge 30-l$. Finally, since 
$0\le l\le 3$ it is apparent that $M_2=15-l\ge l=M_3$. Therefore, Lemma \ref{lemA3} is 
indeed applicable and delivers the bound
$$ \iint_{\mathfrak w} \mathscr F(\alpha,\beta)\,\mathrm d\alpha\,\mathrm d\beta
\ll \iint _{\mathfrak w}|f(\mathrm M_1)^{s-15}f(\mathrm M_2)^{15-l}f(\mathrm M_3)^l| 
\,\mathrm d\alpha\,\mathrm d\beta ,$$
where each of the $\mathrm M_j$ is one of the linear forms $\Lambda_i$, and any two of 
the $\mathrm M_j$ are linearly independent. We now reduce the exponent $s-15$ to $15-l$ 
and then apply H\"older's inequality. Thus
\begin{align*}
\iint_{\mathfrak w} \mathscr F(\alpha,\beta)\,\mathrm d\alpha\,\mathrm d\beta &\ll   
P^{s-30+l}\iint_{\mathfrak w} |f(\mathrm M_1)^{15-l}f(\mathrm M_2)^{15-l} 
f(\mathrm M_3)^l|  \,\mathrm d\alpha\,\mathrm d\beta \\
&\ll \Ups_1^{l/4}\Ups_2^{1-l/4},
\end{align*}
where
\begin{align*}
\Ups_1&=\iint_{\mathfrak w}|f(\mathrm M_1)^{11}f(\mathrm M_2)^{11}
f(\mathrm M_3)^4|  \,\mathrm d\alpha\,\mathrm d\beta,\\
\Ups_2&=\iint_{\mathfrak w} |f(\mathrm M_1)f(\mathrm M_2)|^{15}\,\mathrm d\alpha\,
\mathrm d\beta .
\end{align*}
In this scenario, therefore, we deduce from Lemmata \ref{lem8.1} and \ref{lem8.2} that
\begin{align}
\iint_{\mathfrak w} \mathscr F(\alpha,\beta)\,\mathrm d\alpha\,\mathrm d\beta &\ll 
P^{s-30+l+\varepsilon}\left( P^{18-1/18}\right)^{l/4}\left( P^{22-1/6}\right)^{1-l/4}
\notag \\
&\ll P^{s-8-1/18+\varepsilon}.\label{8.9}
\end{align}

We may now suppose that $l\ge 4$. Then $r_1\le s-15$ and $r_1+r_2\le s-4$. In Lemma 
\ref{lemA3} we now take $J_j$ to be the subset of the set of indices $\{1,2,\ldots ,s\}$ 
counted by $r_j$. We also choose
$$M_\nu=\ldots =M_4=0,\quad M_3=4,\quad M_2=11\quad \text{and}\quad M_1=s-15,$$
and note that the hypothesis $s\ge 26$ ensures that $M_1\ge M_2$. The conditions 
required to apply Lemma \ref{lemA3} are consequently in play, and we deduce that
$$ \iint_{\mathfrak w} \mathscr F(\alpha,\beta)\,\mathrm d\alpha\,\mathrm d\beta \ll 
\iint_{\mathfrak w}|f(\mathrm M_1)|^{s-15}|f(\mathrm M_2)|^{11}
|f(\mathrm M_3)|^{4}\,\mathrm d\alpha\,\mathrm d\beta, $$
where again each of the $\mathrm M_j$ is one of the linear forms $\Lambda_i$, and any 
two of the $\mathrm M_j$ are linearly independent. Here $s-15\ge 11$ by the hypothesis 
$s\ge 26$, and we may estimate excessive copies of $f(\mathrm M_1)$ trivially and apply 
Lemma \ref{lem8.2}. This confirms that
\eqref{8.9} also holds for $l\ge 4$. In particular, we have \eqref{7minor} subject to the 
hypotheses of Theorem \ref{theorem1.1}. This completes the proof of Theorem 
\ref{theorem1.1}.

\section{The proof of theorem \ref{theorem1.2}}
We continue to use the notation introduced in \S\S8 and 9, but now suppose that the 
hypotheses of Theorem \ref{theorem1.2} are met. Hence $s=25$ and $r_1\le s-q_0 \le 9$. 
We also assume that $r_5\ge 1$. Our goal on this occasion is the estimate
\begin{equation}\label{10.1}
\iint_{\mathfrak w} \mathscr F(\alpha,\beta)\,\mathrm d\alpha\,\mathrm d\beta 
\ll P^{17-1/24+\varepsilon}.
\end{equation}
Once this is established, Theorem \ref{theorem1.2} follows in the same way as Theorem 
\ref{theorem1.1} was deduced from \eqref{8.9}.\par

We apply Lemma \ref{lemA3} with $J_j$ the subset of the set of indices $\{1,2,\ldots ,s\}$ 
counted by $r_j$ for $1\le j\le \nu$. Also, we put $m_j=r_j$ for each $j$ and
$$M_\nu=\ldots =M_6=0,\quad M_5=M_4=1, \quad M_3=5\quad \text{and}\quad 
M_2=M_1=9.$$
On recalling that $r_1\le 9$, it is immediate that \eqref{Hyp1} and \eqref{Hyp2} hold. 
Hence, Lemma \ref{lemA3} is applicable, and yields linear forms 
$\mathrm M_1, \ldots,\mathrm M_5$ that are linearly independent in pairs, where each 
$\mathrm M_j$ is one of the $\Lambda_i$, and where
$$  \iint_{\mathfrak w} \mathscr F(\alpha,\beta)\,\mathrm d\alpha\,\mathrm d\beta
\le \iint_{\mathfrak w} |f(\mathrm M_1)^9f(\mathrm M_2)^9f(\mathrm M_3)^5
f(\mathrm M_4)f(\mathrm M_5)|\,\mathrm d\alpha\,\mathrm d\beta. $$
By H\"older's inequality, we find that
$$ \iint_{\mathfrak w} \mathscr F(\alpha,\beta)\,\mathrm d\alpha\,\mathrm d\beta
\le \Ups_3^{1/4} \Ups_4^{3/4},$$
where
\begin{align*}
\Ups_3&= \int_0^1\!\!\int_0^1|f(\mathrm M_1)f(\mathrm M_2)f(\mathrm M_4)
f(\mathrm M_5)|^4 |f(\mathrm M_3)|^2\,\mathrm d\alpha\,\mathrm d\beta,\\
\Ups_4&=\iint_{\mathfrak w}|f(\mathrm M_1)f(\mathrm M_2)|^{32/3}|f(\mathrm M_3)|^6
\, \mathrm d\alpha\, \mathrm d\beta.
\end{align*}
Making use of the bounds supplied by Theorem \ref{thm6.2} and Theorem \ref{thm5.4} 
with $u=32/3$, we therefore infer that
$$\iint_{\mathfrak w} \mathscr F(\alpha,\beta)\,\mathrm d\alpha\,\mathrm d\beta 
\ll P^\varepsilon \left( P^{11}\right)^{1/4}\left( P^{19-1/18}\right)^{3/4}
\ll P^{17-1/24+\varepsilon}.$$
Thus the bound \eqref{10.1} is confirmed, and the proof of Theorem \ref{theorem1.2} is 
complete.\medskip

Finally, we briefly comment on the prospects of reducing the number of variables further. 
Note that the estimates for the minor arcs and for the whole unit square in Theorem 
\ref{thm5.1} coincide for $u=25/3$. Since $\delta(25/3)=0$, therefore, when $s=25$ our 
basic method narrowly fails to be applicable to the system of equations \eqref{1.1}. Further, 
it transpires that each additional variable contributes a factor $P$ to the major arc 
contribution, but only $P^{5/6}$ to the minor arc versions of Theorems \ref{thm5.1} and 
\ref{thm5.2}. As indicated in \S 1 already, it is worth comparing the $18$th moment 
($u=6$) in Theorem \ref{thm5.1} with that in Theorem \ref{thm6.2}, the latter being 
superior by a factor $P^{1/6}$. It transpires that even if it were possible to propagate this 
saving through the moment method, then we would still fail to handle cases of \eqref{1.1} 
with $s=24$, but only by a factor $P^\varepsilon$. However, at this stage, the only 
workable compromise seems to be to apply Theorem \ref{thm6.2} in conjunction with 
Theorems \ref{thm5.1} or \ref{thm5.4}, via H\"older's inequality. If the profile of the 
equations \eqref{1.1} is even more illustrious than in Theorem \ref{theorem1.2}, then one 
can put more weight on the bound stemming from Theorem \ref{thm6.2}. For example, if 
we suppose that $s=24$ and $r_1\le 5$, then $\nu\ge 5$ and $r_5\le 4$, so that in 
hopefully self-explanatory notation, the minor arc contribution can be reduced to something 
of the shape
$$
\iint_{\mathfrak w} \mathscr F(\alpha,\beta)\,\mathrm d\alpha\,\mathrm d\beta 
\ll \iint_{\mathfrak w} |f(\mathrm M_1)^5f(\mathrm M_2)^5f(\mathrm M_3)^5
f(\mathrm M_4)^5f(\mathrm M_5)^4|\,\mathrm d\alpha\,\mathrm d\beta . $$
One may then introduce the identity \eqref{2.4} with $\alpha=\mathrm M_j$ for all 
$1\le j\le 5$ simultaneously. The most difficult term that then arises is that weighted with 
${\sf n}(\mathrm M_1)\cdots{\sf n}(\mathrm M_5)$. A cascade of applications of H\"older's 
inequality together with Theorem \ref{5.1} shows this term to be bounded by 
$$ (\Ups_3)^{3/5}  (J_{11})^{2/5} \ll P^{16+1/15+\varepsilon}, $$
which is quite far from saving another variable.

\bibliographystyle{amsbracket}
\providecommand{\bysame}{\leavevmode\hbox to3em{\hrulefill}\thinspace}

\end{document}